\documentclass[10pt]{amsart}

\usepackage{amsfonts,amssymb}

\input{xy}
\xyoption{all}

\newtheorem{proposition}{Proposition}[section]
\newtheorem{lemma}[proposition]{Lemma}
\newtheorem{corollary}[proposition]{Corollary}
\newtheorem{theorem}[proposition]{Theorem}
\theoremstyle{definition}
\newtheorem{definition}[proposition]{Definition}

\newtheorem{example}[proposition]{Example}

\theoremstyle{remark}
\newtheorem{remark}[proposition]{Remark}

\newcommand{\thlabel}[1]{\label{th:#1}}
\newcommand{\thref}[1]{Theorem~\ref{th:#1}}
\newcommand{\selabel}[1]{\label{se:#1}}
\newcommand{\seref}[1]{Section~\ref{se:#1}}
\newcommand{\sselabel}[1]{\label{sse:#1}}
\newcommand{\sseref}[1]{Subsection~\ref{sse:#1}}
\newcommand{\lelabel}[1]{\label{le:#1}}
\newcommand{\leref}[1]{Lemma~\ref{le:#1}}
\newcommand{\prlabel}[1]{\label{pr:#1}}
\newcommand{\prref}[1]{Proposition~\ref{pr:#1}}
\newcommand{\colabel}[1]{\label{co:#1}}
\newcommand{\coref}[1]{Corollary~\ref{co:#1}}

\newcommand{\relabel}[1]{\label{re:#1}}
\newcommand{\reref}[1]{Remark~\ref{re:#1}}

\newcommand{\exlabel}[1]{\label{ex:#1}}
\newcommand{\exref}[1]{Example~\ref{ex:#1}}
\newcommand{\delabel}[1]{\label{de:#1}}
\newcommand{\deref}[1]{Definition~\ref{de:#1}}

\newcommand{\eqlabel}[1]{\label{eq:#1}}
\newcommand{\equref}[1]{(\ref{eq:#1})}
\def\equuref#1#2{(\ref{eq:#1}.#2)}

\def\ra{\rightarrow}

\def\Id{{\rm Id}}

\newcommand{\Mm}{\mathcal{M}}

\def\ot{\otimes}
\def\va{\varepsilon}

\def\un{\underline}

\def\mf{\mathfrak}

\def\le{\langle}
\def\ri{\rangle}
\def\l{\lambda}

\def\r{\rho}

\def\va{\varepsilon}

\def\ra{\rightarrow}
\def\a{\alpha}
\def\b{\beta}

\def\ov{\overline}
\def\cal{\mathcal}
\def\un{\underline}

\newcommand{\gbrb}{\mbox{$\gbr\gvac{-1}\gnot{\hspace*{-4mm}\bullet}$}}

\newcommand{\YD}{\mathcal{Y\hspace*{-1mm}D}}

\newcommand{\Cc}{\cal C}
\newcommand{\Dc}{\cal D}

\newcommand{\gdbd}{\mbox{$\gdb\gvac{-1}\gnot{\hspace*{-4mm}\vspace*{2mm}\bullet}\gvac{1}$}}
\newcommand{\gevd}{\mbox{$\gev\gvac{-1}\gnot{\hspace*{-4mm}\vspace{-2mm}\bullet}\gvac{1}$}}

\def\equal#1{\smash{\mathop{=}\limits^{#1}}}

\def\equalupdown#1#2{\smash{\mathop{=}\limits^{#1}\limits_{#2}}}
\allowdisplaybreaks[4]

 \newcount\grcalca
 \newcount\grcalcb
 \newcount\grcalcc
 \newcount\grcalcd
 \newcount\grcalce
 \newcount\grcalcf
 \newcount\grcalcg
 \newcount\grcalch
 \newcount\grcalci
 \newcount\grrow
 \newcount\grcolumn
 \newcount\grwidth
 \newcount\factor
 \newcount\hfactor
 \newcount\qfactor
 \newcount\tfactor
 \newcount\sfactor
 \newcount\dfactor
 \hfactor = \factor
 \divide    \hfactor by 2
 \qfactor = \hfactor
 \divide    \qfactor by 2
 \tfactor = \factor
 \divide    \tfactor by 3
 \sfactor = \factor
 \divide    \sfactor by 6
 \dfactor = \factor
 \divide    \dfactor by 12

 \newcommand{\gbeg}[2]{
   \unitlength=1pt
   \grrow = #2
   \grcolumn = 0
   \grcalca = #1
   \grcalcb = #2
   \multiply \grcalca by \factor
   \grwidth = \grcalca
   \multiply \grcalcb by \factor
   \begin{minipage}{\grcalca pt}
   \begin{picture}(\grcalca,\grcalcb)
   \advance \grcalcb by -\factor
   \put(0, \grcalcb){\line(1,0){\grwidth}} }

 \newcommand{\gend}{
   \put(0, \factor){\line(1,0){\grwidth}}
   \end{picture}
   {\vskip2.5ex}
   \end{minipage} }

 \newcommand{\gnl}{
   \advance \grrow by -1
   \grcolumn = 0}

 \newcommand{\gvac}[1]{       
   \advance \grcolumn by #1} 
 
 \newcommand{\gcl}[1]{
   \grcalca = \grcolumn
   \multiply \grcalca by \factor
   \advance \grcalca by \hfactor
   \grcalcb = \grrow
   \multiply \grcalcb by \factor
   \grcalcc = #1
   \multiply \grcalcc by \factor
   \put(\grcalca,\grcalcb) {\line(0,-1){\grcalcc}} 
   \advance \grcolumn by 1}

 \newcommand{\gcn}[4]{
   \grcalca = \grcolumn
   \multiply \grcalca by \factor
   \grcalci = #3
   \multiply \grcalci by \hfactor
   \advance \grcalca by \grcalci
   \grcalcb = \grcolumn
   \multiply \grcalcb by \factor 
   \grcalci = #3
   \advance \grcalci by #4
   \multiply \grcalci by \qfactor
   \advance \grcalcb by \grcalci
   \grcalcc = \grcolumn
   \multiply \grcalcc by \factor
   \grcalci = #4
   \multiply \grcalci by \hfactor
   \advance \grcalcc by \grcalci
   \grcalcd = \grrow
   \multiply \grcalcd by \factor 
   \grcalce = \grrow
   \multiply \grcalce by \factor 
   \grcalci = #2
   \multiply \grcalci by \tfactor
   \advance \grcalce by -\grcalci
   \grcalcf = \grrow
   \multiply \grcalcf by \factor 
   \grcalci = #2
   \multiply \grcalci by \hfactor
   \advance \grcalcf by -\grcalci
   \grcalcg = \grrow
   \multiply \grcalcg by \factor 
   \grcalci = #2
   \multiply \grcalci by \tfactor
   \multiply \grcalci by 2
   \advance \grcalcg by -\grcalci
   \grcalch = \grrow
   \advance \grcalch by -#2
   \multiply \grcalch by \factor 
   \qbezier(\grcalca,\grcalcd)(\grcalca,\grcalce)(\grcalcb,\grcalcf) 
   \qbezier(\grcalcb,\grcalcf)(\grcalcc,\grcalcg)(\grcalcc,\grcalch) 
   \advance \grcolumn by #1}

 \newcommand{\gnot}[1]{
   \grcalca = \grcolumn
   \multiply \grcalca by \factor
   \advance \grcalca by \hfactor
   \grcalcb = \grrow
   \multiply \grcalcb by \factor
   \advance \grcalcb by -\hfactor
   \put(\grcalca,\grcalcb) {\makebox(0,0){$\scriptstyle #1$}} }

 \newcommand{\got}[2]{
   \grcalca = \grcolumn
   \multiply \grcalca by \factor
   \grcalcc = #1
   \multiply \grcalcc by \hfactor
   \advance \grcalca by \grcalcc
   \grcalcb = \grrow
   \multiply \grcalcb by \factor
   \advance \grcalcb by -\tfactor
   \advance \grcalcb by -\tfactor
   \put(\grcalca,\grcalcb){\makebox(0,0)[b]{$#2$}}
   \advance \grcolumn by #1}

 \newcommand{\gob}[2]{
   \grcalca = \grcolumn
   \multiply \grcalca by \factor
   \grcalcc = #1
   \multiply \grcalcc by \hfactor
   \advance \grcalca by \grcalcc
   \put(\grcalca,0){\makebox(0,0)[b]{$#2$}}
   \advance \grcolumn by #1}

 \newcommand{\gmu}{  
   \grcalca = \grcolumn
   \advance \grcalca by 1
   \multiply \grcalca by \factor
   \grcalcb = \grrow
   \multiply \grcalcb by \factor
   \grcalcc = \factor
   \advance \grcalcc by \hfactor
   \put(\grcalca,\grcalcb){\oval(\factor,\grcalcc)[b]}
   \advance \grcalcb by -\hfactor
   \advance \grcalcb by -\qfactor
   \put(\grcalca,\grcalcb) {\line(0,-1){\qfactor}} 
   \advance \grcolumn by 2}

 \newcommand{\gcmu}{   
   \grcalca = \grcolumn
   \advance \grcalca by 1
   \multiply \grcalca by \factor
   \grcalcb = \grrow
   \advance \grcalcb by -1
   \multiply \grcalcb by \factor
   \grcalcc = \factor
   \advance \grcalcc by \hfactor
   \put(\grcalca,\grcalcb){\oval(\factor,\grcalcc)[t]}
   \advance \grcalcb by \factor
   \put(\grcalca,\grcalcb) {\line(0,-1){\qfactor}} 
   \advance \grcolumn by 2}

 \newcommand{\glm}{
   \grcalca = \grcolumn
   \multiply \grcalca by \factor
   \advance \grcalca by \hfactor
   \grcalcb = \grcalca
   \advance \grcalcb by \factor
   \grcalcc = \grrow
   \multiply \grcalcc by \factor
   \grcalcd = \grcalcc
   \advance \grcalcd by -\tfactor
   \grcalce = \grcalcd
   \advance \grcalce by -\tfactor
   \put(\grcalca, \grcalcc){\line(0,-1){\tfactor}}
   \put(\grcalca, \grcalcd){\line(1,0){\factor}}
   \put(\grcalca, \grcalcd){\line(3,-1){\factor}}
   \put(\grcalcb, \grcalcc){\line(0,-1){\factor}}
   \advance \grcolumn by 2}

 \newcommand{\grm}{
   \grcalcb = \grcolumn
   \multiply \grcalcb by \factor
   \advance \grcalcb by \hfactor
   \grcalca = \grcalcb
   \advance \grcalca by \factor
   \grcalcc = \grrow
   \multiply \grcalcc by \factor
   \grcalcd = \grcalcc
   \advance \grcalcd by -\tfactor
   \grcalce = \grcalcd
   \advance \grcalce by -\tfactor
   \put(\grcalca, \grcalcc){\line(0,-1){\tfactor}}
   \put(\grcalca, \grcalcd){\line(-1,0){\factor}}
   \put(\grcalca, \grcalcd){\line(-3,-1){\factor}}
   \put(\grcalcb, \grcalcc){\line(0,-1){\factor}}
   \advance \grcolumn by 2}

 \newcommand{\glcm}{
   \grcalca = \grcolumn
   \multiply \grcalca by \factor
   \advance \grcalca by \hfactor
   \grcalcb = \grcalca
   \advance \grcalcb by \factor
   \grcalcc = \grrow
   \advance \grcalcc by -1
   \multiply \grcalcc by \factor
   \grcalcd = \grcalcc
   \advance \grcalcd by \tfactor
   \grcalce = \grcalcd
   \advance \grcalce by \tfactor
   \put(\grcalca, \grcalcc){\line(0,1){\tfactor}}
   \put(\grcalca, \grcalcd){\line(1,0){\factor}}
   \put(\grcalca, \grcalcd){\line(3,1){\factor}}
   \put(\grcalcb, \grcalcc){\line(0,1){\factor}}
   \advance \grcolumn by 2}

 \newcommand{\grcm}{
   \grcalcb = \grcolumn
   \multiply \grcalcb by \factor
   \advance \grcalcb by \hfactor
   \grcalca = \grcalcb
   \advance \grcalca by \factor
   \grcalcc = \grrow
   \advance \grcalcc by -1
   \multiply \grcalcc by \factor
   \grcalcd = \grcalcc
   \advance \grcalcd by \tfactor
   \grcalce = \grcalcd
   \advance \grcalce by \tfactor
   \put(\grcalca, \grcalcc){\line(0,1){\tfactor}}
   \put(\grcalca, \grcalcd){\line(-1,0){\factor}}
   \put(\grcalca, \grcalcd){\line(-3,1){\factor}}
   \put(\grcalcb, \grcalcc){\line(0,1){\factor}}
   \advance \grcolumn by 2}

 \newcommand{\gwmu}[1]{    
   \grcalca = \grcolumn
   \multiply \grcalca by \factor
   \grcalcd = \hfactor
   \multiply \grcalcd by #1
   \advance \grcalca by \grcalcd
   \grcalcb = \grrow
   \multiply \grcalcb by \factor
   \grcalcc = \factor
   \advance \grcalcc by \hfactor
   \grcalcd = #1
   \advance \grcalcd by -1
   \multiply \grcalcd by \factor
   \put(\grcalca,\grcalcb){\oval(\grcalcd,\grcalcc)[b]}
   \advance \grcalcb by -\hfactor
   \advance \grcalcb by -\qfactor
   \put(\grcalca,\grcalcb) {\line(0,-1){\qfactor}} 
   \advance \grcolumn by #1}

 \newcommand{\gwcm}[1]{   
   \grcalca = \grcolumn
   \multiply \grcalca by \factor
   \grcalcd = \hfactor
   \multiply \grcalcd by #1
   \advance \grcalca by \grcalcd
   \grcalcb = \grrow
   \advance \grcalcb by -1
   \multiply \grcalcb by \factor
   \grcalcc = \factor
   \advance \grcalcc by \hfactor
   \grcalcd = #1
   \advance \grcalcd by -1
   \multiply \grcalcd by \factor
   \put(\grcalca,\grcalcb){\oval(\grcalcd,\grcalcc)[t]}
   \advance \grcalcb by \factor
   \put(\grcalca,\grcalcb) {\line(0,-1){\qfactor}} 
   \advance \grcolumn by #1}

 \newcommand{\gwmuc}[1]{    
   \grcalca = \grcolumn
   \multiply \grcalca by \factor
   \advance \grcalca by \hfactor
   \grcalcb = \grrow
   \multiply \grcalcb by \factor
   \grcalcc = #1
   \advance \grcalcc by -1
   \multiply \grcalcc by \factor
   \put(\grcalca,\grcalcb){\line(1,0){\grcalcc}}
   \advance \grcalca by -\hfactor
   \grcalcd = \hfactor
   \multiply \grcalcd by #1
   \advance \grcalca by \grcalcd
   \grcalcc = \factor
   \advance \grcalcc by \hfactor
   \grcalcd = #1
   \advance \grcalcd by -1
   \multiply \grcalcd by \factor
   \put(\grcalca,\grcalcb){\oval(\grcalcd,\grcalcc)[b]}
   \advance \grcalcb by -\hfactor
   \advance \grcalcb by -\qfactor
   \put(\grcalca,\grcalcb) {\line(0,-1){\qfactor}} 
   \advance \grcolumn by #1}

 \newcommand{\gwcmc}[1]{   
   \grcalca = \grcolumn
   \multiply \grcalca by \factor
   \advance \grcalca by \hfactor
   \grcalcb = \grrow
   \multiply \grcalcb by \factor
   \advance \grcalcb by -\factor
   \grcalcc = #1
   \advance \grcalcc by -1
   \multiply \grcalcc by \factor
   \put(\grcalca,\grcalcb){\line(1,0){\grcalcc}}
   \grcalcd = #1
   \advance \grcalcd by -1
   \multiply \grcalcd by \hfactor
   \advance \grcalca by \grcalcd
   \grcalcc = \factor
   \advance \grcalcc by \hfactor
   \grcalcd = #1
   \advance \grcalcd by -1
   \multiply \grcalcd by \factor
   \put(\grcalca,\grcalcb){\oval(\grcalcd,\grcalcc)[t]}
   \advance \grcalcb by \factor
   \put(\grcalca,\grcalcb) {\line(0,-1){\qfactor}} 
   \advance \grcolumn by #1}

 \newcommand{\gev}{  
   \grcalca = \grcolumn
   \advance \grcalca by 1
   \multiply \grcalca by \factor
   \grcalcb = \grrow
   \multiply \grcalcb by \factor
   \grcalcc = \factor
   \advance \grcalcc by \hfactor
   \put(\grcalca,\grcalcb){\oval(\factor,\grcalcc)[b]}
   \advance \grcolumn by 2}

 \newcommand{\gdb}{   
   \grcalca = \grcolumn
   \advance \grcalca by 1
   \multiply \grcalca by \factor
   \grcalcb = \grrow
   \advance \grcalcb by -1
   \multiply \grcalcb by \factor
   \grcalcc = \factor
   \advance \grcalcc by \hfactor
   \put(\grcalca,\grcalcb){\oval(\factor,\grcalcc)[t]}
   \advance \grcolumn by 2}

 \newcommand{\gwev}[1]{    
   \grcalca = \grcolumn
   \multiply \grcalca by \factor
   \grcalcd = \hfactor
   \multiply \grcalcd by #1
   \advance \grcalca by \grcalcd
   \grcalcb = \grrow
   \multiply \grcalcb by \factor
   \grcalcc = \factor
   \advance \grcalcc by \hfactor
   \grcalcd = #1
   \advance \grcalcd by -1
   \multiply \grcalcd by \factor
   \put(\grcalca,\grcalcb){\oval(\grcalcd,\grcalcc)[b]}
   \advance \grcolumn by #1}

 \newcommand{\gwdb}[1]{   
   \grcalca = \grcolumn
   \multiply \grcalca by \factor
   \grcalcd = \hfactor
   \multiply \grcalcd by #1
   \advance \grcalca by \grcalcd
   \grcalcb = \grrow
   \advance \grcalcb by -1
   \multiply \grcalcb by \factor
   \grcalcc = \factor
   \advance \grcalcc by \hfactor
   \grcalcd = #1
   \advance \grcalcd by -1
   \multiply \grcalcd by \factor
   \put(\grcalca,\grcalcb){\oval(\grcalcd,\grcalcc)[t]}
   \advance \grcolumn by #1}

 \newcommand{\gbr}{
   \grcalca = \grcolumn
   \multiply \grcalca by \factor
   \advance \grcalca by \hfactor
   \grcalcb = \grcalca
   \advance \grcalcb by \hfactor
   \grcalcc = \grcalca
   \advance \grcalcc by \factor
   \grcalcd = \grrow
   \multiply \grcalcd by \factor
   \grcalce = \grcalcd
   \advance \grcalce by -\tfactor
   \grcalcf = \grcalcd
   \advance \grcalcf by -\hfactor
   \grcalcg = \grcalce
   \advance \grcalcg by -\tfactor
   \grcalch = \grcalcd
   \advance \grcalch by -\factor
   \qbezier(\grcalca,\grcalcd)(\grcalca,\grcalce)(\grcalcb,\grcalcf) 
   \qbezier(\grcalcb,\grcalcf)(\grcalcc,\grcalcg)(\grcalcc,\grcalch) 
   \advance \grcalcf by -\dfactor
   \advance \grcalcb by -\sfactor
   \qbezier(\grcalca,\grcalch)(\grcalca,\grcalcg)(\grcalcb,\grcalcf) 
   \advance \grcalcf by \sfactor
   \advance \grcalcb by \tfactor
   \qbezier(\grcalcc,\grcalcd)(\grcalcc,\grcalce)(\grcalcb,\grcalcf) 
   \advance \grcolumn by 2}

 \newcommand{\gibr}{
   \grcalca = \grcolumn
   \multiply \grcalca by \factor
   \advance \grcalca by \hfactor
   \grcalcb = \grcalca
   \advance \grcalcb by \hfactor
   \grcalcc = \grcalca
   \advance \grcalcc by \factor
   \grcalcd = \grrow
   \multiply \grcalcd by \factor
   \grcalce = \grcalcd
   \advance \grcalce by -\tfactor
   \grcalcf = \grcalcd
   \advance \grcalcf by -\hfactor
   \grcalcg = \grcalce
   \advance \grcalcg by -\tfactor
   \grcalch = \grcalcd
   \advance \grcalch by -\factor
   \qbezier(\grcalcc,\grcalcd)(\grcalcc,\grcalce)(\grcalcb,\grcalcf) 
   \qbezier(\grcalcb,\grcalcf)(\grcalca,\grcalcg)(\grcalca,\grcalch) 
   \advance \grcalcf by -\dfactor
   \advance \grcalcb by \sfactor
   \qbezier(\grcalcc,\grcalch)(\grcalcc,\grcalcg)(\grcalcb,\grcalcf) 
   \advance \grcalcf by \sfactor
   \advance \grcalcb by -\tfactor
   \qbezier(\grcalca,\grcalcd)(\grcalca,\grcalce)(\grcalcb,\grcalcf) 
   \advance \grcolumn by 2}

\newcommand{\gsy}{
   \grcalca = \grcolumn
   \multiply \grcalca by \factor
   \advance \grcalca by \hfactor
   \grcalcb = \grcalca
   \advance \grcalcb by \hfactor
   \grcalcc = \grcalca
   \advance \grcalcc by \factor
   \grcalcd = \grrow
   \multiply \grcalcd by \factor
   \grcalce = \grcalcd
   \advance \grcalce by -\tfactor
   \grcalcf = \grcalcd
   \advance \grcalcf by -\hfactor
   \grcalcg = \grcalce
   \advance \grcalcg by -\tfactor
   \grcalch = \grcalcd
   \advance \grcalch by -\factor
   \qbezier(\grcalcc,\grcalcd)(\grcalcc,\grcalce)(\grcalcb,\grcalcf) 
   \qbezier(\grcalcb,\grcalcf)(\grcalca,\grcalcg)(\grcalca,\grcalch) 
   \advance \grcalcf by -\dfactor
   \advance \grcalcb by \sfactor
   \qbezier(\grcalcc,\grcalch)(\grcalcc,\grcalcg)(\grcalcb,\grcalcf) 
   \qbezier(\grcalca,\grcalcd)(\grcalca,\grcalce)(\grcalcb,\grcalcf) 
   \advance \grcolumn by 2}

 \newcommand{\gbrc}{
   \grcalca = \grcolumn
   \multiply \grcalca by \factor
   \advance \grcalca by \hfactor
   \grcalcb = \grcalca
   \advance \grcalcb by \hfactor
   \grcalcc = \grcalca
   \advance \grcalcc by \factor
   \grcalcd = \grrow
   \multiply \grcalcd by \factor
   \grcalce = \grcalcd
   \advance \grcalce by -\tfactor
   \grcalcf = \grcalcd
   \advance \grcalcf by -\hfactor
   \grcalcg = \grcalce
   \advance \grcalcg by -\tfactor
   \grcalch = \grcalcd
   \advance \grcalch by -\factor
   \put(\grcalcb,\grcalcf){\circle{\hfactor}}
   \qbezier(\grcalca,\grcalcd)(\grcalca,\grcalce)(\grcalcb,\grcalcf) 
   \qbezier(\grcalcb,\grcalcf)(\grcalcc,\grcalcg)(\grcalcc,\grcalch) 
   \advance \grcalcf by -\dfactor
   \advance \grcalcb by -\sfactor
   \qbezier(\grcalca,\grcalch)(\grcalca,\grcalcg)(\grcalcb,\grcalcf) 
   \advance \grcalcf by \sfactor
   \advance \grcalcb by \tfactor
   \qbezier(\grcalcc,\grcalcd)(\grcalcc,\grcalce)(\grcalcb,\grcalcf) 
   \advance \grcolumn by 2}

 \newcommand{\gibrc}{
   \grcalca = \grcolumn
   \multiply \grcalca by \factor
   \advance \grcalca by \hfactor
   \grcalcb = \grcalca
   \advance \grcalcb by \hfactor
   \grcalcc = \grcalca
   \advance \grcalcc by \factor
   \grcalcd = \grrow
   \multiply \grcalcd by \factor
   \grcalce = \grcalcd
   \advance \grcalce by -\tfactor
   \grcalcf = \grcalcd
   \advance \grcalcf by -\hfactor
   \grcalcg = \grcalce
   \advance \grcalcg by -\tfactor
   \grcalch = \grcalcd
   \advance \grcalch by -\factor
   \put(\grcalcb,\grcalcf){\circle{\hfactor}}
   \qbezier(\grcalcc,\grcalcd)(\grcalcc,\grcalce)(\grcalcb,\grcalcf) 
   \qbezier(\grcalcb,\grcalcf)(\grcalca,\grcalcg)(\grcalca,\grcalch) 
   \advance \grcalcf by -\dfactor
   \advance \grcalcb by \sfactor
   \qbezier(\grcalcc,\grcalch)(\grcalcc,\grcalcg)(\grcalcb,\grcalcf) 
   \advance \grcalcf by \sfactor
   \advance \grcalcb by -\tfactor
   \qbezier(\grcalca,\grcalcd)(\grcalca,\grcalce)(\grcalcb,\grcalcf) 
   \advance \grcolumn by 2}

 \newcommand{\gu}[1]{
   \grcalca = \grcolumn
   \multiply \grcalca by \factor
   \grcalcd = \hfactor
   \multiply \grcalcd by #1
   \advance \grcalca by \grcalcd
   \grcalcb = \grrow
   \advance \grcalcb by -1
   \multiply \grcalcb by \factor
   \put(\grcalca,\grcalcb) {\line(0,1){\hfactor}} 
   \advance \grcalcb by \hfactor
   \put(\grcalca,\grcalcb) {\circle*{3}}
   \advance \grcolumn by #1}

 \newcommand{\gcu}[1]{
   \grcalca = \grcolumn
   \multiply \grcalca by \factor
   \grcalcd = \hfactor
   \multiply \grcalcd by #1
   \advance \grcalca by \grcalcd
   \grcalcb = \grrow
   \multiply \grcalcb by \factor
   \put(\grcalca,\grcalcb) {\line(0,-1){\hfactor}} 
   \advance \grcalcb by -\hfactor
   \put(\grcalca,\grcalcb) {\circle*{3}}
   \advance \grcolumn by #1}

 \newcommand{\gmp}[1]{
   \grcalca = \grcolumn
   \multiply \grcalca by \factor
   \advance \grcalca by \hfactor
   \grcalcb = \grrow
   \multiply \grcalcb by \factor
   \put(\grcalca,\grcalcb) {\line(0,-1){\dfactor}} 
   \advance \grcalcb by -\factor
   \put(\grcalca,\grcalcb) {\line(0,1){\dfactor}} 
   \advance \grcalcb by \hfactor
   \grcalcc = \factor
   \advance \grcalcc by -\qfactor
   \put(\grcalca,\grcalcb) {\circle{\grcalcc}}
   \put(\grcalca,\grcalcb) {\makebox(0,0){$\scriptstyle #1$}}
   \advance \grcolumn by 1}

 \newcommand{\gbmp}[1]{
   \grcalca = \grcolumn
   \multiply \grcalca by \factor
   \advance \grcalca by \hfactor
   \grcalcb = \grrow
   \multiply \grcalcb by \factor
   \put(\grcalca,\grcalcb) {\line(0,-1){\dfactor}} 
   \advance \grcalcb by -\factor
   \put(\grcalca,\grcalcb) {\line(0,1){\dfactor}} 
   \advance \grcalca by -\hfactor
   \advance \grcalca by \dfactor
   \advance \grcalcb by \dfactor
   \grcalcc = \factor
   \advance \grcalcc by -\sfactor
   \put(\grcalca,\grcalcb) {\framebox(\grcalcc,\grcalcc){$\scriptstyle #1$}}
   \advance \grcolumn by 1}

 \newcommand{\gbmpt}[1]{
   \grcalca = \grcolumn
   \multiply \grcalca by \factor
   \advance \grcalca by \hfactor
   \grcalcb = \grrow
   \multiply \grcalcb by \factor
   \put(\grcalca,\grcalcb) {\line(0,-1){\dfactor}} 
   \advance \grcalcb by -\factor
   \advance \grcalca by -\hfactor
   \advance \grcalca by \dfactor
   \advance \grcalcb by \dfactor
   \grcalcc = \factor
   \advance \grcalcc by -\sfactor
   \put(\grcalca,\grcalcb) {\framebox(\grcalcc,\grcalcc){$\scriptstyle #1$}}
   \advance \grcolumn by 1}

 \newcommand{\gbmpb}[1]{
   \grcalca = \grcolumn
   \multiply \grcalca by \factor
   \advance \grcalca by \hfactor
   \grcalcb = \grrow
   \multiply \grcalcb by \factor
   \advance \grcalcb by -\factor
   \put(\grcalca,\grcalcb) {\line(0,1){\dfactor}} 
   \advance \grcalca by -\hfactor
   \advance \grcalca by \dfactor
   \advance \grcalcb by \dfactor
   \grcalcc = \factor
   \advance \grcalcc by -\sfactor
   \put(\grcalca,\grcalcb) {\framebox(\grcalcc,\grcalcc){$\scriptstyle #1$}}
   \advance \grcolumn by 1}

 \newcommand{\gbmpn}[1]{
   \grcalca = \grcolumn
   \multiply \grcalca by \factor
   \advance \grcalca by \hfactor
   \grcalcb = \grrow
   \multiply \grcalcb by \factor
   \advance \grcalcb by -\factor
   \advance \grcalca by -\hfactor
   \advance \grcalca by \dfactor
   \advance \grcalcb by \dfactor
   \grcalcc = \factor
   \advance \grcalcc by -\sfactor
   \put(\grcalca,\grcalcb) {\framebox(\grcalcc,\grcalcc){$\scriptstyle #1$}}
   \advance \grcolumn by 1}

 \newcommand{\glmptb}{    
   \grcalca = \grcolumn
   \multiply \grcalca by \factor
   \advance \grcalca by \hfactor
   \grcalcb = \grrow
   \multiply \grcalcb by \factor
   \put(\grcalca,\grcalcb) {\line(0,-1){\dfactor}} 
   \advance \grcalcb by -\factor
   \put(\grcalca,\grcalcb) {\line(0,1){\dfactor}} 
   \advance \grcalca by -\hfactor
   \advance \grcalca by \dfactor
   \advance \grcalcb by \dfactor
   \put(\grcalca,\grcalcb) {\line(1,0){\factor}} 
   \advance \grcalcb by \factor
   \advance \grcalcb by -\sfactor
   \put(\grcalca,\grcalcb) {\line(1,0){\factor}} 
   \grcalcc = \factor
   \advance \grcalcc by -\sfactor
   \put(\grcalca,\grcalcb) {\line(0,-1){\grcalcc}} 
   \advance \grcolumn by 1}

 \newcommand{\glmpt}{    
   \grcalca = \grcolumn
   \multiply \grcalca by \factor
   \advance \grcalca by \hfactor
   \grcalcb = \grrow
   \multiply \grcalcb by \factor
   \put(\grcalca,\grcalcb) {\line(0,-1){\dfactor}} 
   \advance \grcalca by -\hfactor
   \advance \grcalca by \dfactor
   \advance \grcalcb by -\dfactor
   \put(\grcalca,\grcalcb) {\line(1,0){\factor}} 
   \advance \grcalcb by -\factor
   \advance \grcalcb by \sfactor
   \put(\grcalca,\grcalcb) {\line(1,0){\factor}} 
   \grcalcc = \factor
   \advance \grcalcc by -\sfactor
   \put(\grcalca,\grcalcb) {\line(0,1){\grcalcc}} 
   \advance \grcolumn by 1}

 \newcommand{\glmpb}{    
   \grcalca = \grcolumn
   \multiply \grcalca by \factor
   \advance \grcalca by \hfactor
   \grcalcb = \grrow
   \multiply \grcalcb by \factor
   \advance \grcalcb by -\factor
   \put(\grcalca,\grcalcb) {\line(0,1){\dfactor}} 
   \advance \grcalca by -\hfactor
   \advance \grcalca by \dfactor
   \advance \grcalcb by \dfactor
   \put(\grcalca,\grcalcb) {\line(1,0){\factor}} 
   \advance \grcalcb by \factor
   \advance \grcalcb by -\sfactor
   \put(\grcalca,\grcalcb) {\line(1,0){\factor}} 
   \grcalcc = \factor
   \advance \grcalcc by -\sfactor
   \put(\grcalca,\grcalcb) {\line(0,-1){\grcalcc}} 
   \advance \grcolumn by 1}

 \newcommand{\glmp}{    
   \grcalca = \grcolumn
   \multiply \grcalca by \factor
   \advance \grcalca by \dfactor
   \grcalcb = \grrow
   \multiply \grcalcb by \factor
   \advance \grcalcb by -\dfactor
   \put(\grcalca,\grcalcb) {\line(1,0){\factor}} 
   \advance \grcalcb by -\factor
   \advance \grcalcb by \sfactor
   \put(\grcalca,\grcalcb) {\line(1,0){\factor}} 
   \grcalcc = \factor
   \advance \grcalcc by -\sfactor
   \put(\grcalca,\grcalcb) {\line(0,1){\grcalcc}} 
   \advance \grcolumn by 1}

 \newcommand{\gcmptb}{    
   \grcalca = \grcolumn
   \multiply \grcalca by \factor
   \advance \grcalca by \hfactor
   \grcalcb = \grrow
   \multiply \grcalcb by \factor
   \put(\grcalca,\grcalcb) {\line(0,-1){\dfactor}} 
   \advance \grcalcb by -\factor
   \put(\grcalca,\grcalcb) {\line(0,1){\dfactor}} 
   \advance \grcalca by -\hfactor
   \advance \grcalcb by \dfactor
   \put(\grcalca,\grcalcb) {\line(1,0){\factor}} 
   \advance \grcalcb by \factor
   \advance \grcalcb by -\sfactor
   \put(\grcalca,\grcalcb) {\line(1,0){\factor}} 
   \advance \grcolumn by 1}

\newcommand{\gmpcu}[1]{
   \grcalca = \grcolumn
   \multiply \grcalca by \factor
   \advance \grcalca by \hfactor
   \grcalcb = \grrow
   \multiply \grcalcb by \factor
   \put(\grcalca,\grcalcb) {\line(0,-1){\dfactor}} 
   \advance \grcalcb by -\factor
   \advance \grcalcb by \hfactor
   \grcalcc = \factor
   \advance \grcalcc by -\qfactor
   \put(\grcalca,\grcalcb) {\circle{\grcalcc}}
   \put(\grcalca,\grcalcb) {\makebox(0,0){$\scriptstyle #1$}}
   \advance \grcolumn by 1}
   
\newcommand{\gmpu}[1]{
   \grcalca = \grcolumn
   \multiply \grcalca by \factor
   \advance \grcalca by \hfactor
   \grcalcb = \grrow
   \multiply \grcalcb by \factor
   \advance \grcalcb by -\factor
   \put(\grcalca,\grcalcb) {\line(0,1){\dfactor}} 
   \advance \grcalcb by \hfactor
   \grcalcc = \factor
   \advance \grcalcc by -\qfactor
   \put(\grcalca,\grcalcb) {\circle{\grcalcc}}
   \put(\grcalca,\grcalcb) {\makebox(0,0){$\scriptstyle #1$}}
   \advance \grcolumn by 1}      

 \newcommand{\gcmpt}{    
   \grcalca = \grcolumn
   \multiply \grcalca by \factor
   \advance \grcalca by \hfactor
   \grcalcb = \grrow
   \multiply \grcalcb by \factor
   \put(\grcalca,\grcalcb) {\line(0,-1){\dfactor}} 
   \advance \grcalcb by -\factor
   \advance \grcalca by -\hfactor
   \advance \grcalcb by \dfactor
   \put(\grcalca,\grcalcb) {\line(1,0){\factor}} 
   \advance \grcalcb by \factor
   \advance \grcalcb by -\sfactor
   \put(\grcalca,\grcalcb) {\line(1,0){\factor}} 
   \advance \grcolumn by 1}

 \newcommand{\gcmpb}{    
   \grcalca = \grcolumn
   \multiply \grcalca by \factor
   \advance \grcalca by \hfactor
   \grcalcb = \grrow
   \multiply \grcalcb by \factor
   \advance \grcalcb by -\factor
   \put(\grcalca,\grcalcb) {\line(0,1){\dfactor}} 
   \advance \grcalca by -\hfactor
   \advance \grcalcb by \dfactor
   \put(\grcalca,\grcalcb) {\line(1,0){\factor}} 
   \advance \grcalcb by \factor
   \advance \grcalcb by -\sfactor
   \put(\grcalca,\grcalcb) {\line(1,0){\factor}} 
   \advance \grcolumn by 1}

 \newcommand{\gcmp}{    
   \grcalca = \grcolumn
   \multiply \grcalca by \factor
   \grcalcb = \grrow
   \multiply \grcalcb by \factor
   \advance \grcalcb by -\factor
   \advance \grcalcb by \dfactor
   \put(\grcalca,\grcalcb) {\line(1,0){\factor}} 
   \advance \grcalcb by \factor
   \advance \grcalcb by -\sfactor
   \put(\grcalca,\grcalcb) {\line(1,0){\factor}} 
   \advance \grcolumn by 1}

 \newcommand{\grmptb}{    
   \grcalca = \grcolumn
   \multiply \grcalca by \factor
   \advance \grcalca by \hfactor
   \grcalcb = \grrow
   \multiply \grcalcb by \factor
   \put(\grcalca,\grcalcb) {\line(0,-1){\dfactor}} 
   \advance \grcalcb by -\factor
   \put(\grcalca,\grcalcb) {\line(0,1){\dfactor}} 
   \advance \grcalca by \hfactor
   \advance \grcalca by -\dfactor
   \advance \grcalcb by \dfactor
   \put(\grcalca,\grcalcb) {\line(-1,0){\factor}} 
   \advance \grcalcb by \factor
   \advance \grcalcb by -\sfactor
   \put(\grcalca,\grcalcb) {\line(-1,0){\factor}} 
   \grcalcc = \factor
   \advance \grcalcc by -\sfactor
   \put(\grcalca,\grcalcb) {\line(0,-1){\grcalcc}} 
   \advance \grcolumn by 1}

 \newcommand{\grmpt}{    
   \grcalca = \grcolumn
   \multiply \grcalca by \factor
   \advance \grcalca by \hfactor
   \grcalcb = \grrow
   \multiply \grcalcb by \factor
   \put(\grcalca,\grcalcb) {\line(0,-1){\dfactor}} 
   \advance \grcalca by \hfactor
   \advance \grcalca by -\dfactor
   \advance \grcalcb by -\dfactor
   \put(\grcalca,\grcalcb) {\line(-1,0){\factor}} 
   \advance \grcalcb by -\factor
   \advance \grcalcb by \sfactor
   \put(\grcalca,\grcalcb) {\line(-1,0){\factor}} 
   \grcalcc = \factor
   \advance \grcalcc by -\sfactor
   \put(\grcalca,\grcalcb) {\line(0,1){\grcalcc}} 
   \advance \grcolumn by 1}

 \newcommand{\grmpb}{    
   \grcalca = \grcolumn
   \multiply \grcalca by \factor
   \advance \grcalca by \hfactor
   \grcalcb = \grrow
   \multiply \grcalcb by \factor
   \advance \grcalcb by -\factor
   \put(\grcalca,\grcalcb) {\line(0,1){\dfactor}} 
   \advance \grcalca by \hfactor
   \advance \grcalca by -\dfactor
   \advance \grcalcb by \dfactor
   \put(\grcalca,\grcalcb) {\line(-1,0){\factor}} 
   \advance \grcalcb by \factor
   \advance \grcalcb by -\sfactor
   \put(\grcalca,\grcalcb) {\line(-1,0){\factor}} 
   \grcalcc = \factor
   \advance \grcalcc by -\sfactor
   \put(\grcalca,\grcalcb) {\line(0,-1){\grcalcc}} 
   \advance \grcolumn by 1}

 \newcommand{\grmp}{    
   \grcalca = \grcolumn
   \multiply \grcalca by \factor
   \advance \grcalca by \factor
   \advance \grcalca by -\dfactor
   \grcalcb = \grrow
   \multiply \grcalcb by \factor
   \advance \grcalcb by -\dfactor
   \put(\grcalca,\grcalcb) {\line(-1,0){\factor}} 
   \advance \grcalcb by -\factor
   \advance \grcalcb by \sfactor
   \put(\grcalca,\grcalcb) {\line(-1,0){\factor}} 
   \grcalcc = \factor
   \advance \grcalcc by -\sfactor
   \put(\grcalca,\grcalcb) {\line(0,1){\grcalcc}} 
   \advance \grcolumn by 1}

 \newcommand{\gwmuh}[3]{    
   \grcalca = \grcolumn
   \multiply \grcalca by \factor
   \grcalcb = #2
   \advance \grcalcb by #3
   \multiply \grcalcb by \qfactor
   \advance \grcalca by \grcalcb
   \grcalcb = \grrow
   \multiply \grcalcb by \factor
   \grcalcc = #3
   \advance \grcalcc by -#2
   \multiply \grcalcc by \hfactor
   \grcalcd = \factor
   \advance \grcalcd by \hfactor
   \put(\grcalca,\grcalcb){\oval(\grcalcc,\grcalcd)[b]}
   \grcalca = \grcolumn
   \multiply \grcalca by \factor
   \grcalcc = #1
   \multiply \grcalcc by \hfactor
   \advance \grcalca by \grcalcc
   \advance \grcalcb by -\hfactor
   \advance \grcalcb by -\qfactor
   \put(\grcalca,\grcalcb) {\line(0,-1){\qfactor}} 
   \advance \grcolumn by #1}

 \newcommand{\gwcmh}[3]{   
   \grcalca = \grcolumn
   \multiply \grcalca by \factor
   \grcalcb = #2
   \advance \grcalcb by #3
   \multiply \grcalcb by \qfactor
   \advance \grcalca by \grcalcb
   \grcalcb = \grrow
   \advance \grcalcb by -1
   \multiply \grcalcb by \factor
   \grcalcc = #3
   \advance \grcalcc by -#2
   \multiply \grcalcc by \hfactor
   \grcalcd = \factor
   \advance \grcalcd by \hfactor
   \put(\grcalca,\grcalcb){\oval(\grcalcc,\grcalcd)[t]}
   \grcalca = \grcolumn
   \multiply \grcalca by \factor
   \grcalcc = #1
   \multiply \grcalcc by \hfactor
   \advance \grcalca by \grcalcc
   \advance \grcalcb by \factor
   \put(\grcalca,\grcalcb) {\line(0,-1){\qfactor}} 
   \advance \grcolumn by #1}

 \newcommand{\gsbox}[1]{
   \grcalca = \grcolumn
   \multiply \grcalca by \factor
   \grcalcb = \grrow
   \multiply \grcalcb by \factor
   \advance \grcalcb by -\factor
   \grcalcc = #1
   \multiply \grcalcc by \factor
   \grcalcd = \factor
   \put(\grcalca,\grcalcb){\framebox(\grcalcc,\grcalcd){}}}

\begin{document}
\title[Ribbon quasi-Hopf algebras]
{On sovereign, balanced and ribbon quasi-Hopf algebras}
\author{Daniel Bulacu}
\address{Faculty of Mathematics and Computer Science, University
of Bucharest, Str. Academiei 14, RO-010014 Bucharest 1, Romania}
\email{daniel.bulacu@fmi.unibuc.ro}
\author{Blas Torrecillas}
\address{Department of Mathematics\\
Universidad de Almer\'{\i}a\\
E-04071 Almer\'{\i}a, Spain}
\email{btorreci@ual.es}
\subjclass[2010]{Primary 16W30; Secondaries 18D10; 16S34}
\thanks{Work supported by the project MTM2017-86987-P "Anillos, modulos y algebra de Hopf".  
The first author thanks the University of Almeria (Spain) for its warm hospitality. 
The authors also thank Bodo Pareigis for sharing his "diagrams" program.}

\begin{abstract}
We introduce the notions of sovereign, spherical and balanced quasi-Hopf algebra. We investigate the connections between 
these, as well as their connections with the class of pivotal, involutory and ribbon quasi-Hopf algebras, respectively. 
Examples of balanced and ribbon quasi-Hopf algebras are obtained from a sort of double construction which associates 
to a braided category (resp. rigid braided) a balanced (resp. ribbon) one.    
\end{abstract}

\keywords{}
\maketitle

\section{Introduction}
The theory of monoidal categories plays an important role in (quantum) topology, a domain with applications   
in knot theory, link theory, the classification of manifolds, algebraic geometry, etc. In some cases, the monoidal categories 
are identified with categories of (co)representations of a Hopf like algebra; for instance, some modular categories identify 
with categories of finite-dimensional comodules of a certain weak Hopf algebra (see \cite[Theorem 1.1]{hp}), while fusion categories 
for which each simple object has an integer Frobenius-Perron dimension are precisely the categories of representations of a finite-dimensional 
quasi-Hopf algebra (see \cite[Theorem 8.33]{eno}). This led to a growing interest for the study of those categories of 
(co)modules which are monoidal, rigid, sovereign, pivotal, spherical, braided, balanced, ribbon or modular, to name a few. 
Many such monoidal categories were invented with topological applications in mind. 
For instance, spherical categories were introduced in \cite{barret1} in order to generalize the Turaev-Viro state sum model invariant of a closed piecewise-linear $3$-manifold, and the main sources of them are categories of representations of involutory 
Hopf algebras and of quantised enveloping algebras at a root of unity; similarly, ribbon categories give rise to link invariants, 
and in particular ribbon Hopf algebras give rise to a topological invariant of knots and links is the $3$-sphere (see \cite{rt}). 
Monoidal categories drawn also attention in physics, algebra and computer science.  

This paper deals with the study of the category of representations of a quasi-Hopf algebra $H$ in the general framework of monoidal 
categories mentioned above. Otherwise stated, we deal with categories $\Cc$ of modules over a $k$-algebra $H$ for which the 
forgetful functor to the category of $k$-vector spaces is quasi-monoidal and, moreover, rigid when it is restricted to the category of 
finite-dimensional $H$-representations, ${}_H{\cal M}^{\rm fd}$ (see \cite{btRecon}). We study first when ${}_H{\cal M}^{\rm fd}$ is sovereign, that is rigid such that the left and right duality functors coincide as monoidal functors; as sovereign categories identify with the pivotal ones (according to \cite{feYett}, see also \thref{pivissovereign}) it comes up that a quasi-Hopf algebra $H$ is sovereign if and only if it is pivotal, and by \cite{bctInv} the latter is equivalent to the existence of a kind of grouplike element in $H$ that 
defines the square of the antipode as an inner automporphism of $H$ (see \prref{sovereinstrquasi}). Then we get for free necessary and sufficient conditions for ${}_H{\cal M}^{\rm fd}$ to be spherical, i.e. a pivotal category for which the left and right traces of an  endomorphism in ${}_H{\cal M}^{\rm fd}$ are equal. We should point out that, as in the Hopf algebra case, particular examples of 
spherical quasi-Hopf algebras are obtained from involutory quasi-Hopf algebras (see \thref{invisspherical}) and their 
quantum doubles. This leads also to a positive answer for a question raised in \cite{bctInv}: $H$ is involutory 
if and only if so is its quantum double $D(H)$ (see \coref{invQDqHa}). 

Starting with \seref{balancedcatqHa} we move to the braided case. It is well understood by now that ribbon categories are balanced, spherical, 
sovereign, etc. categories, so they are good sources for constructing various topological invariants. With these topological 
applications in mind, Kassel and Turaev \cite{kasturaev} associate to any rigid monoidal category $\Cc$ a ribbon one, denoted by $\Dc(\Cc)$. 
Their construction extends the center construction due to Drinfeld (unpublished), Joyal and Street \cite{jscentre} and 
Majid \cite{majcentre}, a construction that associates to a monoidal category a braided one. As observed by 
Drabant \cite{drabant}, the idea behind of the construction of Kassel and Turaev can be used also to obtain balanced categories 
(or braided sovereign categories according to \thref{ribbonvssovereign}, a result owing to Deligne \cite{deligne}) 
from monoidal ones. In particular, any quasitriangular (QT for short) bialgebra (resp. Hopf algebra) gives rise 
to a balanced bialgebra (resp. ribbon Hopf algebra).     

In \prref{balancedfrombraided} we generalize the construction of Drabant, by "replacing" the centre category with an arbitrary braided category $(\Cc, c)$. The same thing we do in \seref{ribboncategqHa}, where in \thref{ribbonfrombraided} we generalize 
the construction of Kassel and Turaev to an arbitrary rigid braided category $(\Cc, c)$. 
Our approaches allow to simplify the computations in the case when we apply these constructions to the category of 
representations of a quasi-bialgebra or a quasi-Hopf algebra $H$ (in general complicated by the apparitions of the reassociator $\Phi$ of $H$ and of the triple that defines the antipode of it). More exactly, to any QT quasi-bialgebra $(H, R)$ we associate a balanced one, denoted by 
$H[\theta, \theta^{-1}]$, in such a way that ${}_{H[\theta, \theta^{-1}]}{\cal M}$ and ${\cal B}({}_H{\cal M}, c)$ are isomorphic as 
balanced categories (see \prref{balancedqHacase}), where in general by ${\cal B}(\Cc, c)$ we denote the balanced category associated 
through our construction to the braided category $(\Cc, c)$. 
Similarly, to a QT quasi-Hopf algebra $(H, R)$ we can associate a ribbon quasi-Hopf algebra $(H(\theta), R)$ is such a way that 
${\cal R}({}_H{\cal M}^{\rm fd}, c)$ identifies as a ribbon category with ${}_{H(\theta)}{\cal M}^{\rm fd}$ 
(see \thref{qtreprasrepresoverribbon}), where in general ${\cal R}(\Cc, c)$ stands for the ribbon category associated through our construction to the rigid braided category $(\Cc, c)$. Furthermore, when we apply this to the category of finite dimensional Yetter-Drinfeld modules 
over an arbitrary finite-dimensional quasi-Hopf algebra $H$ we get for free that ${\cal R}({}_H{\cal YD}^{H{\rm fd}}, {\mf c})$ and 
${}_{D(H)(\theta)}{\cal M}^{\rm fd}$ are isomorphic as ribbon categories, where $D(H)$ is the quantum double of $H$ 
(see \coref{ribbdrinffromadj}). We conclude by mentioning, one more time, that our constructions apply to any (rigid) braided category 
(so for instance to ${}_H{\cal M}$ with $(H, R)$ a QT quasi-bialgebra or quasi-Hopf algebra), hence not necessarily equals a centre 
of a (rigid) monoidal category, leading thus to new link invariants.            
\section{Preliminaries}\selabel{prelim}
\subsection{Quasi-bialgebras and quasi-Hopf algebras}\sselabel{quasiHopfalg}
We work over a field $k$. All algebras, linear
spaces, etc. will be over $k$; unadorned $\ot $ means $\ot_k$.
Following Drinfeld \cite{dri}, a quasi-bialgebra is
a four-tuple $(H, \Delta , \va , \Phi )$ where $H$ is
an associative algebra with unit,
$\Phi$ is an invertible element in $H\ot H\ot H$, and
$\Delta :\ H\ra H\ot H$ and $\va :\ H\ra k$ are algebra
homomorphisms satisfying the identities
\begin{eqnarray}
&&(\Id_H \ot \Delta )(\Delta (h))=
\Phi (\Delta \ot \Id_H)(\Delta (h))\Phi ^{-1},\eqlabel{q1}\\
&&(\Id_H \ot \va )(\Delta (h))=h~~,~~
(\va \ot \Id_H)(\Delta (h))=h,\eqlabel{q2}
\end{eqnarray}
for all $h\in H$, where
$\Phi$ is a $3$-cocycle, in the sense that
\begin{eqnarray}
&&(1\ot \Phi)(\Id_H\ot \Delta \ot \Id_H)
(\Phi)(\Phi \ot 1)\nonumber\\
&&\hspace*{1.5cm}
=(\Id_H\ot \Id_H \ot \Delta )(\Phi )
(\Delta \ot \Id_H \ot \Id_H)(\Phi),\eqlabel{q3}\\
&&(\Id \ot \va \ot \Id_H)(\Phi)=1\ot 1.\eqlabel{q4}
\end{eqnarray}
The map $\Delta$ is called the coproduct or the
comultiplication, $\va $ is the counit, and $\Phi $ is the
reassociator. As for Hopf algebras we denote $\Delta (h)=h_1\ot h_2$,
but since $\Delta$ is only quasi-coassociative we adopt the
further convention (summation understood):
$$
(\Delta \ot \Id_H)(\Delta (h))=h_{(1, 1)}\ot h_{(1, 2)}\ot h_2~~,~~
(\Id_H\ot \Delta )(\Delta (h))=h_1\ot h_{(2, 1)}\ot h_{(2,2)},
$$
for all $h\in H$. We will
denote the tensor components of $\Phi$
by capital letters, and the ones of
$\Phi^{-1}$ by lower case letters, namely
\begin{eqnarray*}
&&\Phi=X^1\ot X^2\ot X^3=T^1\ot T^2\ot T^3=
V^1\ot V^2\ot V^3=\cdots\\
&&\Phi^{-1}=x^1\ot x^2\ot x^3=t^1\ot t^2\ot t^3=
v^1\ot v^2\ot v^3=\cdots
\end{eqnarray*}
$H$ is called a quasi-Hopf
algebra if, moreover, there exists an
anti-morphism $S$ of the algebra
$H$ and elements $\a , \b \in
H$ such that, for all $h\in H$, we
have:
\begin{eqnarray}
&&
S(h_1)\a h_2=\va(h)\a
~~{\rm and}~~
h_1\b S(h_2)=\va (h)\b,\eqlabel{q5}\\ 
&&X^1\b S(X^2)\a X^3=1
~~{\rm and}~~
S(x^1)\a x^2\b S(x^3)=1.\eqlabel{q6}
\end{eqnarray}

Our definition of a quasi-Hopf algebra is different from the
one given by Drinfeld \cite{dri} in the sense that we do not
require the antipode to be bijective. In the case where $H$ is finite dimensional
or quasitriangular, bijectivity of the antipode follows from the other axioms,
see \cite{bc1} and \cite{bn3}, so the two definitions are equivalent.

It is well-known that the antipode of a Hopf algebra is an 
anti-morphism of coalgebras. For a quasi-Hopf algebra, we have something close that follows from the 
following general result due to Drinfeld, see \cite[Lemma 2]{dri}.

\begin{lemma}\lelabel{prelantcoalgqHa}
Let $H$ be a quasi-Hopf algebra and $A$ a $k$-algebra. Suppose that there exist 
an algebra map $f: H\ra A$, an anti-algebra map $g: H\ra A$, and elements 
$\rho, \sigma\in A$ such that 
\begin{eqnarray}
&&g(h_1)\rho f(h_2)=\va(h)\rho,~~f(h_1)\sigma g(h_2)=\va(h)\sigma,~~
\forall~~h\in H,~\eqlabel{qHaf5}\\
&&f(X^1)\sigma g(X^2)\rho f(X^3)=1_A,~~g(x^1)\rho f(x^2)\sigma g(x^3)=1_A.\eqlabel{qHaf6}
\end{eqnarray}
If $\ov{g}: H\ra A$ is another anti-algebra map 
and $\ov{\rho}$, $\ov{\sigma}\in A$ are such that \equref{qHaf5} and \equref{qHaf6} 
hold for $f$, $\ov{g}$, $\ov{\sigma}$ and $\ov{\rho}$ as well, then there exists a unique 
invertible element $F\in A$ such that $\ov{\rho}=F\rho$, $\ov{\sigma}=\sigma F^{-1}$ 
and $\ov{g}(h)=Fg(h)F^{-1}$, for all $h\in H$. Furthermore, 
\[
F=\ov{g}(x^1)\ov{\rho}f(x^2)\sigma g(x^3)~~\mbox{with}~~
F^{-1}=g(x^1)\rho f(x^2)\ov{\sigma}~\ov{g}(x^3).
\]
\end{lemma}

If we define $\gamma, \delta\in H\ot H$ by
\begin{eqnarray}
&&\gamma=S(x^1X^2)\a x^2X^3_1\ot S(X^1)\a x^3X^3_2
\equal{(\ref{eq:q3},\ref{eq:q5})}
 S(X^2x^1_2)\a X^3x^2\ot S(X^1x^1_1)\a x^3,
\eqlabel{gamma}\\
&&\delta=X^1_1x^1\b S(X^3)\ot X^1_2x^2\b S(X^2x^3)
\equal{(\ref{eq:q3},\ref{eq:q5})} x^1\b S(x^3_2X^3)\ot x^2X^1\b S(x^3_1X^2)\eqlabel{delta}
\end{eqnarray}
and apply \leref{prelantcoalgqHa} to $A=H\ot H$, $f=\Delta$, $g=\Delta\circ S: H\ra H\ot H$, 
$\rho=\Delta(\a)$, $\sigma=\Delta(\b)$, 
$\ov{g}=(S\ot S)\circ \Delta^{\rm cop}: H\ra H\ot H$, $\ov{\rho}=\gamma$ and 
$\ov{\sigma}=\delta$, then there exists a unique invertible element 
$f=f^1\ot f^2\in H\ot H$, called the Drinfeld twist or the gauge transformation, 
\begin{eqnarray}
f&=&(S\ot S)(\Delta ^{\rm op}(x^1)) \gamma \Delta (x^2\b
S(x^3))\eqlabel{f}~~\mbox{with}\\ 
f^{-1}&=&\Delta (S(x^1)\a x^2) \delta
(S\ot S)(\Delta^{\rm cop}(x^3)),\eqlabel{g}
\end{eqnarray}
such that $\va(f^1)f^2=\va(f^2)f^1=1$ and 
\begin{eqnarray} 
&&f\Delta (S(h))f^{-1}= (S\ot S)(\Delta ^{\rm cop}(h)),~\forall~h\in H,\eqlabel{ca}\\
&&f\Delta (\a )=\gamma~~,~~\Delta (\b )f^{-1}=\delta.\eqlabel{gdf}
\end{eqnarray}

Furthermore, \equref{ca} is part of the fact that $S: H^{\rm op, cop}\ra H_f$ 
is a quasi-Hopf algebra morphism,  where $H_f$ is the twisting of the quasi-Hopf algebra
$H$ by the Drinfeld twist $f$. We refer to \cite{dri} for more details; as far as we are concerned, 
record only that this property of $S$ implies 
\begin{eqnarray}
&&X^1g^1_1G^1\ot X^2g^1_2G^2\ot X^3g^2=g^1S(X^3)\ot g^2_1G^1S(X^2)\ot g^2_2G^2S(X^1),\eqlabel{phig}\\
&&f^1\b S(f^2)=S(\a)~~\mbox{and}~~S(g^1)\a g^2=S(\b),\label{l3}
\end{eqnarray} 
where $f=f^1\ot f^2$ and $f^{-1}:=g^1\ot g^2=G^1\ot G^2$. 

\subsection{Sovereign and spherical categories}\sselabel{moncateg}
For the definition of a monoidal category $\Cc$ and related topics we refer to \cite{kas, maj}. 
Usually, for a monoidal category $\Cc$, we denote by $\ot$ the tensor product, by $\un{1}$ the unit object, and 
by $a, l, r$ the associativity constraint and the left and right unit constraints, respectively. The monoidal category 
$(\Cc, \ot , \un{1}, a, l, r)$ is called strict if all the natural isomorphisms $a, l$ and $r$ are defined by 
identity morphisms in $\Cc$.

If $(\Cc, \ot , \un{1}, a, l, r)$ is a monoidal category, by  
\[
\ov{\Cc}:=(\Cc, \ov{\ot}:=\ot\circ \tau,~ \ov{a}=(\ov{a}_{X, Y, Z}:=a^{-1}_{Z, Y, X})_{X, Y, Z\in \Cc}, 
\un{1}, \ov{l}:=r, \ov{r}:=l)
\]
we denote the reverse monoidal category associated to $\Cc$, where $\tau : \Cc\times \Cc\ra \Cc\times \Cc$ stands for the twist functor, 
that is $\tau(X, Y)=(Y, X)$, for any $X, Y\in \Cc$, and $\tau(f, g)=(g, f)$, for 
any morphisms $X\stackrel{f}{\rightarrow}X'$ and $Y\stackrel{g}{\rightarrow}Y'$ in $\Cc$.

When $\Cc$ is, moreover, 
left rigid we denote by $X^*$ the left dual of an object $X$ of $\Cc$ and by 
$
{\rm coev}_X=
{{\footnotesize
\gbeg{2}{3}
\got{2}{\un{1}}\gnl
\gdb\gnl
\gob{1}{X}\gob{1}{X^*}
\gend}}
: \un{1}\ra X\ot X^*
$ and 
$ 
{\rm ev}_X=
{{\footnotesize
\gbeg{2}{3}
\got{1}{X^*}\got{1}{X}\gnl
\gev\gnl
\gob{2}{\un{1}}
\gend}}: X^*\ot X\ra \un{1} 
$ 
the corresponding coevaluation and evaluation morphisms. When $\Cc$ is strict we have that 
\begin{equation}\eqlabel{rrigid}
{{\footnotesize
\gbeg{3}{4}
\gvac{2}\got{1}{X}\gnl
\gdb\gcl{1}\gnl
\gcl{1}\gev\gnl
\gob{1}{X}
\gend }}=
{{\footnotesize
\gbeg{1}{3}
\got{1}{X}\gnl
\gcl{1}\gnl
\gob{1}{X}
\gend}}
\hspace{3mm}{\rm and}\hspace{3mm}
{{\footnotesize
\gbeg{3}{4}
\got{1}{X^*}\gnl
\gcl{1}\gdb\gnl
\gev\gcl{1}\gnl
\gvac{2}\gob{1}{X^*}
\gend }}=
{{\footnotesize
\gbeg{1}{3}
\got{1}{X^*}\gnl
\gcl{1}\gnl
\gob{1}{X^*}
\gend}}\hspace{2mm}.
\end{equation}

If $\Cc$ is left rigid 
we have a well defined functor $(-)^*: \Cc\ra \Cc^{\rm op}$ that maps $X$ to $X^*$ and a morphism $f$ to $f^*$, the transpose of $f$. 
Here $\Cc^{\rm op}$ is the opposite category associated to $\Cc$. Furthermore, $(-)^*$ is a strong monoidal functor 
if it is regarded as a functor from $\Cc$ to $\ov{\Cc^{\rm op}}:=(\Cc^{\rm op}, \ot^{\rm op}:=
\ot\circ \tau, \un{1}, (a_{Z, Y, X})_{X, Y, Z\in \Cc}, r^{-1}, l^{-1})$, where $\tau: \Cc\times \Cc\ra \Cc\times \Cc$ is the 
switch functor. The functor $(-)^*$ is called the left dual functor, and its strong monoidal structure is mostly determined by 
\begin{equation}\eqlabel{firstisomdual}
\l_{X, Y}=
{{\footnotesize
\gbeg{6}{6}
\got{1}{(X\ot Y)^*}\gnl
\gcl{1}\gvac{1}\gwdb{4}\gnl
\gcl{1}\gvac{1}\gcl{1}\gdb\gcl{1}\gnl
\gcl{1}\gsbox{3}\gvac{1}\gnot{X\ot Y}\gvac{2}\gcl{1}\gcl{1}\gnl
\gwev{3}\gvac{1}\gcl{1}\gcl{1}\gnl
\gvac{4}\gob{1}{Y^*}\gob{1}{X^*}
\gend }}
\hspace{1mm},\hspace{1mm}
\l_{X, Y}^{-1}=
{{\footnotesize
\gbeg{7}{7}
\got{1}{Y^*}\got{1}{X^*}\gnl
\gcl{1}\gcl{1}\gvac{1}\gwdb{3}\gnl
\gcl{1}\gcl{1}\gsbox{3}\gvac{1}\gnot{X\ot Y}\gvac{2}\gcl{1}\gnl
\gcl{1}\gcl{1}\gcl{1}\gcl{1}\gvac{1}\gcl{1}\gnl
\gcl{1}\gev\gcl{1}\gvac{1}\gcl{1}\gnl
\gwev{4}\gvac{1}\gcl{1}\gnl
\gvac{5}\gob{1}{(X\ot Y)^*}
\gend}},
\end{equation}   
in the sense that $\l:=(\l_{X, Y}: (X\ot Y)^*\ra Y^*\ot X^*)_{X, Y\in \Cc}$ is a natural isomorphism from $(-)^*\circ \ot\ra ((-)^*\ot (-)^*)\circ \ot^{\rm op}$  
such that its inverse $\l^{-1}:=(\l^{-1}_{X, Y})_{X, Y\in \Cc}$ satisfies the conditions in \cite[Definition XI.4.1]{kas}. Here, and in what follows, 
\[
{\rm coev}_{X\ot Y}=
{{\footnotesize
\gbeg{5}{5}
\gvac{2}\got{2}{\un{1}}\gnl
\gwdb{4}\gnl
\gsbox{2}\gnot{\hspace*{5mm}X\ot Y}\gvac{3}\gcl{1}\gnl
\gcl{1}\gcl{1}\gvac{1}\gcl{1}\gnl
\gob{1}{X}\gob{1}{Y}\gvac{1}\gob{1}{(X\ot Y)^*}
\gend
}}\hspace*{3mm}\mbox{and}\hspace*{3mm}
{\rm ev}_{X\ot Y}=
{{\footnotesize
\gbeg{5}{5}
\gvac{1}\got{1}{(X\ot Y)^*}\gvac{1}\got{1}{X}\got{1}{Y}\gnl
\gvac{1}\gcl{1}\gvac{1}\gcl{1}\gcl{1}\gnl
\gvac{1}\gcl{1}\gvac{1}\gsbox{2}\gnot{\hspace*{5mm}X\ot Y}\gnl
\gvac{1}\gwev{3}\gnl
\gvac{2}\gob{2}{\un{1}}
\gend }}
\]
is the diagrammatic notation for the coevaluation and evaluation morphisms of $X\ot Y$. 
 
A monoidal category $\Cc$ is right rigid if $\ov{\Cc}$ is left rigid. If this is the case, the right dual of an object $X$ of $\Cc$ is denoted by ${}^*X$, 
and by 
${\rm coev}'_X:=
{{\footnotesize
\gbeg{2}{3}
\got{2}{\un{1}}\gnl
\gdbd\gnl
\gob{1}{{}^*X}\gob{1}{X}
\gend }}: \un{1}\ra {}^*X\ot X$ 
and 
$
{\rm ev}'_X:=
{{\footnotesize
\gbeg{2}{3}
\got{1}{X}\got{1}{{}^*X}\gnl
\gevd\gnl
\gob{2}{\un{1}}
\gend}}: X\ot {}^*X\ra \un{1}
$ 
we denote the corresponding coevaluation and evaluation morphisms. Then, if $\Cc$ is strict,  
\begin{equation}\eqlabel{evcoev}
{{\footnotesize
\gbeg{3}{4}
\got{1}{X}\gnl
\gcl{1}\gdbd\gnl
\gev\gvac{-1}\gnot{\hspace*{-4mm}\vspace{-2mm}\bullet}\gvac{1}\gcl{1}\gnl
\gob{5}{X}
\gend }}=
{{\footnotesize
\gbeg{1}{3}
\got{1}{X}\gnl
\gcl{1}\gnl
\gob{1}{X}
\gend }}
\hspace{3mm}\mbox{~~and~~}\hspace{3mm}
{{\footnotesize
\gbeg{3}{4}
\got{5}{{}^*X}\gnl
\gdbd\gvac{2}\gcl{1}\gnl
\gcl{1}\gevd\gnl
\gob{1}{{}^*X}
\gend }}=
{{\footnotesize
\gbeg{1}{3}
\got{1}{{}^*X}\gnl
\gcl{1}\gnl
\gob{1}{{}^*X}
\gend }}\hspace{2mm}.
\end{equation} 

As in the left rigid case, 
we have a strong monoidal functor ${}^*(-):\Cc\ra \ov{\Cc^{\rm op}}$, called the right dual functor. Note that ${}^*(-)$ is just 
$(-)^*$ considered for $\ov{\Cc}$ instead of $\Cc$. 

\begin{definition}\delabel{monsovcateg}
A monoidal category is called sovereign if it is left and right rigid
such that the corresponding left and right dual functors $(-)^*, {}^*(-):\Cc\ra\ov{\Cc^{opp}}$ are
equal as strong monoidal functors.
\end{definition}

If $\Cc$ is left rigid, the double dual functor $(-)^{**}:=((-)^*)^*:\ {\cal C}\to {\cal C}$ is strong monoidal, too. 

\begin{definition}
Let $\Cc$ be a a left rigid monoidal category and denote by $\Id_\Cc$ the identity functor of $\Cc$. A 
pivotal structure on ${\cal C}$ is a monoidal natural isomorphism $i$ between 
the strong monoidal functors ${\rm Id}_\Cc$ and $(-)^{**}$. Such a pair $(\Cc, i)$ is called pivotal category.
\end{definition}

It seems that the equivalence between sovereign and pivotal notions goes back to Freyd and Yetter \cite{feYett}. 
For the sake of completeness and also for further use we sketch below this equivalence. 

\begin{theorem}\thlabel{pivissovereign}
Let $\Cc$ be a left rigid monoidal category. Then $\Cc$ admits a pivotal structure if and only if it is sovereign.
\end{theorem} 
\begin{proof}
Let $\Cc$ be a sovereign category (in particular, it is also right rigid). For all $X\in \Cc$ 
we have an isomorphism $\theta_X: X\ra ({}^*X)^*$,
\begin{eqnarray}
&&\hspace*{-5mm}
\theta _X:\ X
\stackrel{r_X^{-1}}{\longrightarrow}X\ot \un{1}
\stackrel{\Id_X\ot {\rm coev}_{{}^*X}}{\longrightarrow}X\ot ({}^*X\ot ({}^*X)^*)\nonumber\\
&&\hspace*{1cm}
\stackrel{a^{-1}_{X, {}^*X, ({}^*X)^*}}{\longrightarrow}(X\ot {}^*{X})\ot ({}^*X)^*
\stackrel{{\rm ev}'_X\ot {\Id_{({}^*X)^*}}}{\longrightarrow}\un{1}\ot ({}^*X)^*
\stackrel{l_{({}^*X)^*}}{\longrightarrow}({}^*X)^*.\eqlabel{theta}
\end{eqnarray} 
Since $\Cc$ is sovereign we have $(-)^*={}^*(-)$, so $({}^*X)^*=X^{**}$. Thus we have $\theta_X: X\ra X^{**}$. 
One can easily see that 
$\theta=(\theta_X)_{X\in {\rm Ob}(\Cc)}: \Id_\Cc\ra (-)^{**}$ is a natural isomorphism. It is, moreover, 
a natural monoidal transformation, and so provides a pivotal structure on $\Cc$.

Conversely, let $i$ be a pivotal structure on $\Cc$, and for all $V\in \Cc$ define 
\begin{eqnarray*}
&&{\rm ev}'_V: V\ot V^*\stackrel{i_V\ot \Id_{V^*}}{\longrightarrow}V^{**}\ot V^*\stackrel{{\rm ev}_{V^*}}{\longrightarrow}\un{1},~~
\mbox{and}\\
&&{\rm coev}'_V: \un{1}\stackrel{{\rm coev}_{V^*}}{\longrightarrow}V^*\ot V^{**}\stackrel{\Id_{V^*}\ot i^{-1}_V}{\longrightarrow}
V^*\ot V.
\end{eqnarray*}
It is immediate that $(V^*, {\rm ev}'_V, {\rm coev}'_V)$ is a right dual for $V$ in $\Cc$, and so $\Cc$ is right rigid, too. 
With respect to this right duality we have that the left and right duality functors 
coincide as monoidal functors, and therefore $\Cc$ is sovereign. 
\end{proof}

For a sovereign (or, equivalently, for a pivotal) category $\Cc$ one can define the left and right traces 
${\rm tr}_l(f)$, ${\rm tr}_r(f)\in {\rm End}_\Cc(\un{1})$ of an endomorphism $f\in {\rm End}_\Cc(V)$ by 
\begin{equation}\eqlabel{defoflrtraces}
{\rm tr}_l(f):=
\gbeg{2}{5}
\got{2}{\un{1}}\gnl
\gdbd\gnl
\gcl{1}\gmp{f}\gnl
\gev\gnl
\gob{2}{\un{1}}
\gend
~~\mbox{and}~~
{\rm tr}_r(f):=
\gbeg{2}{5}
\got{2}{\un{1}}\gnl
\gdb\gnl
\gmp{f}\gcl{1}\gnl
\gevd\gnl
\gob{2}{\un{1}}
\gend.
\end{equation}

The definition of a spherical category in terms of a pivotal structure was given in \cite[Definition 2.4]{barret1}. 
If we deal with the sovereign property instead, it reduces to the following.   

\begin{definition}
A sovereign category is called spherical if ${\rm tr}_l(f)={\rm tr}_r(f)$, for all $V\in \Cc$ and $f\in {\rm End}_\Cc(V)$. 
\end{definition}

Note that ${\rm tr}_{l/r}(f)={\rm tr}_{r/l}(f^*)$, for all $f\in {\rm End}_\Cc(V)$, and therefore a sovereign ($\equiv$ pivotal) 
category is spherical if and only if ${\rm tr}_l(f)={\rm tr}_l(f^*)$ or ${\rm tr}_r(f)={\rm tr}_r(f^*)$, for all 
$V\in \Cc$ and $f\in {\rm End}_\Cc(V)$. This is possible since in any sovereign category we have, for all $V\in \Cc$, 
\[
\gbeg{3}{4}
\got{1}{V^{**}}\gnl
\gcl{1}\gdbd\gnl
\gev\gcl{1}\gnl
\gvac{2}\gob{1}{V}
\gend
=
\gbeg{3}{4}
\gvac{2}\got{1}{V^{**}}\gnl
\gdb\gcl{1}\gnl
\gcl{1}\gevd\gnl
\gob{1}{V}
\gend
\hspace*{0.7cm}\mbox{and}\hspace*{0.7cm}
\gbeg{3}{4}
\gvac{2}\got{1}{V}\gnl
\gdbd\gvac{2}\gcl{1}\gnl
\gcl{1}\gev\gnl
\gob{1}{V^{**}}
\gend
=
\gbeg{3}{4}
\got{1}{V}\gnl
\gcl{1}\gdb\gnl
\gevd\gvac{2}\gcl{1}\gnl
\gvac{2}\gob{1}{V^{**}}
\gend~~.
\] 
\subsection{Balanced and ribbon categories}
Roughly speaking, a monoidal category $\Cc$ is braided if it has a braiding $c$, i.e. a family of natural isomorphisms 
$c_{X, Y}: X\ot Y\ra Y\ot X$ satisfying two hexagon axioms, see \cite[XIII.1]{kas}. Any braiding $c$ obeys the categorical version of 
the Yang-Baxter equation. Namely, for any objects $X, Y$ and $Z$ of ${\cal C}$, we have   
\begin{equation}\eqlabel{qybe}
{{\footnotesize
\gbeg{3}{5}
\got{1}{X}\got{1}{Y}\got{1}{Z}\gnl
\gbr\gcl{1}\gnl
\gcl{1}\gbr\gnl
\gbr\gcl{1}\gnl
\gob{1}{Z}\gob{1}{Y}\gob{1}{X}
\gend }}=
{{\footnotesize
\gbeg{3}{5}
\got{1}{X}\got{1}{Y}\got{1}{Z}\gnl
\gcl{1}\gbr\gnl
\gbr\gcl{1}\gnl
\gcl{1}\gbr\gnl
\gob{1}{Z}\gob{1}{Y}\gob{1}{X}
\gend}}.
\end{equation} 
Here $c_{X, Y}=
\gbeg{2}{3}
\got{1}{X}\got{1}{Y}\gnl
\gbr\gnl
\gob{1}{Y}\gob{1}{X}
\gend$ 
is the notation for the braiding $c$ of $\Cc$; similarly, we denote by 
$ 
\gbeg{2}{3}
\got{1}{Y}\got{1}{X}\gnl
\gibr\gnl
\gob{1}{X}\gob{1}{Y}
\gend$ 
the inverse morphism of $c_{X, Y}$ in $\Cc$.  

Let $H$ be a quasi-bialgebra. It is well known that ${}_H{\cal M}$ is braided if and only if 
there exists an invertible element $R=R^1\ot R^2=r^1\ot r^2\in H\ot H$ (formal notation, summation 
implicitly understood) such that the following relations hold:
\begin{eqnarray}
(\Delta \ot \Id_H)(R)&=&X^2R^1x^1Y^1\ot X^3x^3r^1Y^2\ot
X^1R^2x^2r^2Y^3,\eqlabel{qt1}\\ 
(\Id_H\ot \Delta )(R)&=&x^3R^1X^2r^1y^1\ot x^1X^1r^2y^2\ot x^2R^2X^3y^3,\eqlabel{qt2}\\
\Delta ^{\rm cop}(h)R&=&R\Delta (h),~\forall~h\in H.\eqlabel{qt3}
\end{eqnarray}
We say in this case that $H$ is a quasitriangular, QT for short, quasi-bialgebra. Note that 
the braiding $c$ on ${}_H{\cal M}$ defined by $R$ as above is given by
\begin{equation}\eqlabel{braiQT}
c_{X, Y}(x\ot y)=R^2\cdot y\ot R^1\cdot x,~~\forall~~x\in X\in {}_H{\cal M},~~
y\in Y\in {}_H{\cal M}.
\end{equation}  
When we refer to a QT quasi-bialgebra or quasi-Hopf algebra we 
always indicate the $R$-matrix $R$ that produces the QT structure by pointing out the couple $(H, R)$. 
Also, if $t$ denotes a permutation of $\{1, 2, 3\}$, then we set
$\Phi _{t(1)t(2)t(3)}=X^{t^{-1}(1)}\otimes X^{t^{-1}(2)}\otimes
X^{t^{-1}(3)}$, and by $R_{ij}$ we denote the element obtained by acting with $R$ non-trivially in the $i^{th}$ and $j^{th}$
positions of $H\otimes H\otimes H$. 

For $(H, R)$ a QT quasi-Hopf algebra, $u$ is the element of $H$ defined by    
\begin{equation}\eqlabel{elmu}
u=S(R^2x^2\b S(x^3))\a R^1x^1.
\end{equation}
By \cite{bn3}, $u$ is an invertible element of $H$ and the following equalities hold:
\begin{eqnarray}\eqlabel{sinau}
&&\hspace*{-1cm}
S^2(h)=uhu^{-1},~\forall~h\in H,\eqlabel{ssinau}\\
&&\hspace*{-1cm}
S(\a )u=S(R^2)\a R^1,\eqlabel{extr}\\
&&\hspace*{-1cm}
R^1\b S(R^2)=S(\b u),\label{fu18}\\
&&\hspace*{-10mm}
\Delta(u)=
(R_{21}R)^{-1}f^{-1}(S\ot S)(f_{21})(u\ot u),\eqlabel{qrib7}\\
&&\hspace*{-10mm}
\Delta (S(u))=(R_{21}R)^{-1}(S(u)\ot S(u))(S\ot S)(f^{-1}_{21})f.\eqlabel{qrib6}
\end{eqnarray}

We move now to the balanced/ribbon case. 

\begin{definition}
(i) A a braided category $(\Cc, c)$ is called balanced if there exists a natural isomorphism 
$\eta=(\eta_V: V\ra V)_{V\in {\rm Ob}(\Cc)}$, called twist on $\Cc$,  
such that for all $V, W\in \Cc$,
\begin{equation}\eqlabel{ribboncateg1}
\eta _{V\otimes W}=(\eta _V\otimes \eta _W)\circ c_{W,V}\circ c_{V,W}.
\end{equation}

(ii) A balanced category $(\Cc, c, \eta)$ is called ribbon if, in addition, $\Cc$ is left rigid and 
\begin{equation}\eqlabel{ribboncateg2}
\eta _{V^*}=(\eta _V)^*
\end{equation}
for all $V\in \Cc$. If this is the case then $\eta$ is called a ribbon twist on $\Cc$. 
\end{definition}

For $H$ a quasi-bialgebra, when ${}_H{\cal M}$ is a balanced/ribbon category was studied in \cite{bpv}. 

\begin{proposition}\prlabel{whenQTisribbon}
Let $(H, R)$ be a finite dimensional QT quasi-Hopf algebra, so ${}_H\Mm^{\rm fd}$ is a rigid braided category. 
Set $R=R^1\ot R^2$ and $R_{21}=R^2\ot R^1$. Then ${}_H\Mm^{\rm fd}$ is, moreover, 

(i) a balanced category if and only if there exists an invertible central element $\eta\in H$ 
such that
\begin{equation}
\Delta(\eta)=(\eta\ot \eta)R_{21}R;\eqlabel{ribbonqHa2}
\end{equation}

(ii) a ribbon category if and only if there is an element $\eta\in H$ 
as in (i) satisfying also the condition 
\begin{equation}
S(\eta)=\eta.\eqlabel{ribbonqHa1}
\end{equation}
\end{proposition}
\begin{proof}
Since $H$ is finite dimensional we can regard $H\in {}_H\Mm^{\rm fd}$ via its multiplication. 

(i) Suppose that $(\eta_V: V\ra V)_V$, indexed by $V\in {{}_H\Mm^{\rm fd}}$, defines a balanced structure 
on ${}_H\Mm^{\rm fd}$. Then $\eta_H$ is completely determined by $\eta:=\eta(1_H)$. Actually, 
\begin{equation}\eqlabel{etaforribbonqHa}
\eta_V(v)=\eta\cdot v,~\forall~v\in V.
\end{equation} 
Thus \equref{ribboncateg1} is satisfied for all $V, W\in {}_H\Mm^{\rm fd}$ if and only 
if \equref{ribbonqHa2} holds. Furthermore, $\eta_V$ is an isomorphism for any $V\in {}_H{\cal M}^{\rm fd}$ if and only if $\eta$ is invertible, 
and $\eta_V$ is left $H$-linear, for all 
$V\in {}_H\Mm^{\rm fd}$, if and only if $\eta$ is a central element of $H$.

(ii) For the ribbon case, \equref{ribboncateg2} is equivalent to \equref{ribbonqHa1}. 
\end{proof}

The next definitions are imposed by the characterization in \prref{whenQTisribbon}. 

\begin{definition}\delabel{defribbQHA}
(i) We call a QT quasi-bialgebra or quasi-Hopf algebra $(H, R)$ balanced 
if there exists an invertible central element $\eta\in H$ satisfying \equref{ribbonqHa2}.

(ii) A QT quasi-Hopf algebra $(H, R)$ is called a 
ribbon quasi-Hopf algebra if there exists an invertible central element $\eta\in H$ 
satisfying \equref{ribbonqHa2} and \equref{ribbonqHa1}.
\end{definition}

The following result is \cite[Proposition 2.5]{drabant}. It says that if $\Cc$ is left rigid then $\Cc$ is balanced if and only if 
there is a natural transformation $\eta$ satisfying \equref{ribboncateg1}, i.e. the fact that $\eta$ is as well a natural isomorphism is automatic; 
furthermore, the inverse of the square of the twist $\eta$ is completely determined 
by the left rigid braided structure of the category.  

\begin{proposition}\prlabel{balancedfaarfromribbon}
Let $(\Cc, c)$ be a left rigid category and $\eta: \Id_\Cc\ra \Id_\Cc$ a natural transformation satisfying \equref{ribboncateg1}. Then, 
for all $V\in \Cc$,  $\eta_V$ is an isomorphism in $\Cc$ with inverse given by 
\begin{equation}\eqlabel{invofetadualbases}
\eta^{-1}_V=:
{{\footnotesize
\gbeg{3}{7}
\gvac{2}\got{1}{V}\gnl
\gdb\gcl{2}\gnl
\gcl{1}\gmp{\eta_{V^*}}\gnl
\gbr\gcl{1}\gnl
\gcl{1}\gibr\gnl
\gev\gcl{1}\gnl
\gvac{2}\gob{1}{V}
\gend}}
=
{{\footnotesize
\gbeg{3}{7}
\got{1}{V}\gnl
\gcl{1}\gdb\gnl
\gcl{1}\gcl{1}\gmp{\eta_{V^*}}\gnl
\gcl{1}\gbr\gnl
\gbr\gcl{1}\gnl
\gev\gcl{1}\gnl
\gvac{2}\gob{1}{V}
\gend
}},
\end{equation}
and therefore $(\Cc, c, \eta)$ is balanced. Consequently, $(\Cc, c, \eta)$ is ribbon if and only if, for all $V\in \Cc$,   
\begin{equation}\eqlabel{msquaretwist}
\eta^{-2}_V=({\rm ev}_V\ot \Id_V)(\Id_{V^*}\ot c^{-1}_{V, V})(c_{V, V^*}{\rm coev}_V\ot \Id_V). 
\end{equation}
\end{proposition}

Any left/rigid braided category is braided, so any ribbon category is rigid. Furthermore, for a ribbon category the right rigid structure  
can be constructed from the left rigid structure, braiding and twist in such a way that the left and 
right dual functors coincide. Thus the choice of the left duals in the definition of a ribbon category is irrelevant. 

More generally, if $(\Cc, c, \eta)$ is a left rigid balanced category, $V$ is an object of $\Cc$ and $V^*$ is the left dual of $V$ in $\Cc$  
with evaluation and coevaluation morphisms ${\rm ev}_V: V^*\ot V\ra \un{1}$ and ${\rm coev}_V: \un{1}\ra V\ot V^*$,  
then 
\begin{equation}\eqlabel{rightdualitysovereign}
(V^*, {\rm ev}'_V:={\rm ev}_Vc_{V,V^*}(\eta_V\ot \Id_{V^*}), {\rm coev}'_V:=(\eta_{V^*}\ot \Id_V)c_{V, V^*}{\rm coev}_V)
\end{equation} 
is a right dual for $V$ in $\Cc$. Furthermore, with respect to it the left and right dual functors coincide as strong monoidal 
functors, and therefore any left rigid balanced category is sovereign. In particular, any ribbon category is sovereign; note that in this case, 
from $\eta_{V^*}=(\eta_V)^*$ it follows that 
${\rm coev}'_V$ in \equref{rightdualitysovereign} can be restated as 
\begin{equation}\eqlabel{rightdualitysovereign2coev}
{\rm coev}'_V=
(\Id_{V^*}\ot \eta_V)c_{V, V^*}{\rm coev}_V. 
\end{equation}

Actually, we can characterize ribbon categories in terms of sovereign categories. That rigid braided balanced categories are actually 
braided sovereign categories was proved by Deligne \cite{deligne}. The other statements of the theorem below were taken form \cite[Proposition A.4]{henr}. 

\begin{theorem}\thlabel{ribbonvssovereign}
Let $(\Cc, c)$ be a left rigid braided category. Then 

(i) $(\Cc, c)$ is balanced if and only if $\Cc$ is sovereign;

(ii) $(\Cc, c)$ is ribbon if and only if it is sovereign and with respect to the rigid structure given by the fact that $\Cc$  
is sovereign we either have  
\begin{equation}\eqlabel{addcondforribbon}
{{\footnotesize
\gbeg{3}{5}
\gvac{2}\got{1}{V}\gnl
\gdbd\gvac{2}\gcl{1}\gnl
\gcl{1}\gbr\gnl
\gev\gcl{1}\gnl
\gvac{2}\gob{1}{V}
\gend
}}
=
{{\footnotesize
\gbeg{3}{5}
\got{1}{V}\gnl
\gcl{1}\gdb\gnl
\gbr\gcl{1}\gnl
\gcl{1}\gevd\gnl
\gob{1}{V}
\gend}}
~~\mbox{or}~~
{{\footnotesize
\gbeg{3}{5}
\got{1}{V^*}\gnl
\gcl{1}\gdbd\gnl
\gbr\gcl{1}\gnl
\gcl{1}\gev\gnl
\gob{1}{V^*}
\gend}}
=
{{\footnotesize
\gbeg{3}{5}
\gvac{2}\got{1}{V^*}\gnl
\gdb\gcl{1}\gnl
\gcl{1}\gbr\gnl
\gevd\gvac{2}\gcl{1}\gnl
\gvac{2}\gob{1}{V^*}
\gend}},~~\forall~V\in \Cc.
\end{equation}
\end{theorem}
\begin{proof}
(i) We have just seen that a left rigid balanced category is sovereign. For the converse, 
if $(\Cc, c)$ is a braided sovereign category define, for all $V\in \Cc$, 
\begin{equation}\eqlabel{twistfrombrad&sovereign} 
\eta_V:={{\footnotesize
\gbeg{3}{5}
\gvac{2}\got{1}{V}\gnl
\gdbd\gvac{2}\gcl{1}\gnl
\gcl{1}\gbr\gnl
\gev\gcl{1}\gnl
\gvac{2}\gob{1}{V}
\gend
}}~~,~~
\eta^{-1}_V:=
{{\footnotesize
\gbeg{3}{5}
\got{1}{V}\gnl
\gcl{1}\gdb\gnl
\gibr\gcl{1}\gnl
\gcl{1}\gevd\gnl
\gob{1}{V}
\gend
}}~~;~~
{\bf \theta}_V=
{{\footnotesize
\gbeg{3}{5}
\got{1}{V}\gnl
\gcl{1}\gdb\gnl
\gbr\gcl{1}\gnl
\gcl{1}\gevd\gnl
\gob{1}{V}
\gend}}~~,~~
\theta^{-1}_V=
{{\footnotesize
\gbeg{3}{5}
\gvac{2}\got{1}{V}\gnl
\gdbd\gvac{2}\gcl{1}\gnl
\gcl{1}\gibr\gnl
\gev\gcl{1}\gnl
\gvac{2}\gob{1}{V}
\gend
}}.
\end{equation}
Then $\eta:=(\eta_V)_{V\in {\rm Ob}(\Cc)}$ and $\theta:=(\theta_V)_{V\in {\rm Ob}(\Cc)}$ are twists on $\Cc$ with inverses 
defined by $\eta^{-1}:=(\eta^{-1}_V)_{V\in {\rm Ob}(\Cc)}$ and $\theta^{-1}:=(\theta^{-1}_V)_{V\in {\rm Ob}(\Cc)}$, respectively. Hence 
$(\Cc, c, \eta)$ and $(\Cc, c, \theta)$ are braided balanced categories. 

(ii) By \cite[Proposition A.4]{henr} we know that (ii) is equivalent to the first equality in \equref{addcondforribbon}. So we only have to show that, for all $V\in \Cc$,  
$\eta_{V^*}=(\eta_V)^*$ if and only if the second equality in \equref{addcondforribbon} holds. To this end, note that \equref{rightdualitysovereign} 
and the naturality of $c$ imply  
\[
{{\footnotesize
\gbeg{3}{6}
\gvac{2}\got{1}{V^*}\gnl
\gdb\gcl{1}\gnl
\gmp{\eta_V}\gcl{1}\gcl{1}\gnl
\gbr\gcl{1}\gnl
\gcl{1}\gevd\gnl
\gob{1}{V^*}
\gend
}}
=
{{\footnotesize
\gbeg{5}{6}
\gvac{4}\got{1}{V^*}\gnl
\gdbd\gvac{2}\gdb\gcl{1}\gnl
\gcl{1}\gbr\gcl{1}\gcl{1}\gnl
\gev\gbr\gcl{1}\gnl
\gvac{2}\gcl{1}\gevd\gnl
\gvac{2}\gob{1}{V^*}
\gend
}}
=
{{\footnotesize
\gbeg{5}{6}
\gvac{4}\got{1}{V^*}\gnl
\gdbd\gvac{4}\gcl{2}\gnl
\gcl{1}\gcn{1}{1}{1}{5}\gnl
\gcl{1}\gdb\gevd\gnl
\gev\gcl{1}\gnl
\gvac{2}\gob{1}{V^*}
\gend}}
=\Id_{V^*},
\]
and therefore
\begin{eqnarray*}
&&
{{\footnotesize
\gbeg{3}{5}
\gvac{2}\got{1}{V^*}\gnl
\gdb\gcl{1}\gnl
\gcl{1}\gbr\gnl
\gevd\gvac{2}\gcl{1}\gnl
\gvac{2}\gob{1}{V^*}
\gend}}
=
{{\footnotesize
\gbeg{3}{6}
\gvac{2}\got{1}{V^*}\gnl
\gdb\gcl{2}\gnl
\gsbox{2}\gnot{\hspace*{5mm}\eta_{V\ot V^*}}\gnl
\gcl{1}\gbr\gnl
\gevd\gvac{2}\gcl{1}\gnl
\gvac{2}\gob{1}{V^*}
\gend
}}
=
{{\footnotesize
\gbeg{3}{8}
\gvac{2}\got{1}{V^*}\gnl
\gdb\gcl{1}\gnl
\gbr\gcl{1}\gnl
\gbr\gcl{1}\gnl
\gmp{\eta_V}\gmp{\eta_{V^*}}\gcl{1}\gnl
\gcl{1}\gbr\gnl
\gevd\gvac{2}\gcl{1}\gnl
\gvac{2}\gob{1}{V^*}
\gend
}}
=
{{\footnotesize
\gbeg{3}{8}
\gvac{2}\got{1}{V^*}\gnl
\gdb\gcl{1}\gnl
\gmp{\eta_V}\gcl{1}\gcl{1}\gnl
\gbr\gcl{1}\gnl
\gbr\gcl{1}\gnl
\gcl{1}\gbr\gnl
\gevd\gvac{2}\gmp{\eta_{V^*}}\gnl
\gvac{2}\gob{1}{V^*}
\gend
}}
=
{{\footnotesize
\gbeg{3}{6}
\gvac{2}\got{1}{V^*}\gnl
\gdb\gcl{1}\gnl
\gmp{\eta_V}\gcl{1}\gcl{1}\gnl
\gbr\gcl{1}\gnl
\gmp{\eta_{V^*}}\gevd\gnl
\gob{1}{V^*}
\gend}}
=\eta_{V^*},\\ 
&&
(\eta_V)^*=
{{\footnotesize
\gbeg{5}{6}
\got{1}{V^*}\gnl
\gcl{1}\gdbd\gvac{2}\gdb\gnl
\gcl{1}\gcl{1}\gbr\gcl{3}\gnl
\gcl{1}\gev\gcl{1}\gnl
\gwev{4}\gnl
\gvac{4}\gob{1}{V^*}
\gend
}}
=
{{\footnotesize
\gbeg{5}{5}
\got{1}{V^*}\gnl
\gcl{1}\gdbd\gvac{2}\gdb\gnl
\gcl{1}\gibr\gcl{1}\gcl{2}\gnl
\gev\gev\gcl{1}\gnl
\gvac{4}\gob{1}{V^*}
\gend
}}
=
{{\footnotesize
\gbeg{3}{5}
\got{1}{V^*}\gnl
\gcl{1}\gdbd\gnl
\gcl{1}\gibr\gnl
\gev\gcl{1}\gnl
\gvac{2}\gob{1}{V^*}
\gend
}}
=
{{\footnotesize
\gbeg{3}{5}
\got{1}{V^*}\gnl
\gcl{1}\gdbd\gnl
\gbr\gcl{1}\gnl
\gcl{1}\gev\gnl
\gob{1}{V^*}
\gend
}}~,
\end{eqnarray*}
and this finishes the proof.
\end{proof}
Hence, for a braided sovereign category $\Cc$ we have two twists $\eta$ and $\theta$ which are equal if and only if one of them 
provides a ribbon structure on $\Cc$. If this is the case then $\Cc$ endowed with the pivotal structure produced by the ribbon twist 
$\eta=\theta$ is a spherical category. Note that the converse is also true, provided that $\Cc$ is semisimple; 
see \cite[Proposition A.4]{henr}.  
\section{Sovereign and spherical quasi-Hopf algebras}\selabel{sovereignqHa}
\setcounter{equation}{0}
If $H$ is a quasi-bialgebra, the category ${}_H{\cal M}$ of left $H$-representations is monoidal. If $U, V$ are left 
$H$-modules then the tensor product between $U$ and $V$ is the tensor product over $k$ equipped with the left $H$-module 
structure given by $\Delta$, i.e. $h\cdot (u\ot v)=h_1\cdot u\ot h_2\cdot v$, for all $h\in H$, $u\in U$ and $v\in V$. 
The associativity constraint on ${}_H{\cal M}$ is the following:
for $U, V, W\in {}_H{\cal M}$, $a_{U, V, W}: (U\ot V)\ot W\to U\ot (V\ot W)$ is given by
\[
a_{U, V, W}((u\ot v)\ot w)=X^1\cdot u\ot (X^2\cdot v\ot X^3\cdot w).
\]
The unit object is $k$ considered as a left $H$-module via $\va$, the counit of $H$. The left and right unit constraints are 
the same as for the category ${}_k{\cal M}$ of $k$-vector space.  

Let ${}_H{\cal M}^{\rm fd}$ be the full subcategory of ${}_H{\cal M}$ consisting of finite dimensional vector spaces. Then 
${}_H{\cal M}^{\rm fd}$ is a category with left duality, see \reref{sovereinstrleftmodqHa} below. 
Our goal is to see when ${}_H{\cal M}^{\rm fd}$ is a 
sovereign (resp. spherical) category, in the sense of \deref{monsovcateg}. By the general results presented in \sseref{moncateg}, 
it comes up that sovereign structures on ${}_H{\cal M}^{\rm fd}$ are in a one to one correspondence 
with the pivotal ones, and that the latter are completely determined by certain elements of $H$.  

\begin{proposition}\prlabel{sovereinstrquasi}
Let $H$ be a finite dimensional quasi-Hopf algebra over a field $k$. Then we
have a bijective correspondence between 

$(i)$ pivotal structures $i$ on ${\cal C}={}_H{\cal M}^{\rm fd}$;

$(ii)$ sovereign structures on ${\cal C}={}_H{\cal M}^{\rm fd}$; 
 
$(iii)$ invertible elements ${\mf g}_i\in H$ satisfying 
\begin{eqnarray}
&&S^2(h)={\mf g}_i^{-1}h{\mf g}_i,~\forall~h\in H,\label{prop1}\\
&&\Delta({\mf g}_i)=({\mf g}_i\ot {\mf g}_i)(S\ot S)(f_{21}^{-1})f,\label{pivstr2} 
\end{eqnarray} 
where $f=f^1\ot f^2$ is the Drinfeld twist defined in \equref{f} and 
$f_{21}=f^2\ot f^1$.  
\end{proposition}

\begin{proof}
The bijection between (i) and (ii) is established by \thref{pivissovereign}, while the equivalence between (i) and (iii) was established in 
\cite[Proposition 4.2]{bctInv}. 
\end{proof}

\begin{definition}
A sovereign quasi-Hopf algebra is a quasi-Hopf algebra (not necessarily finite dimensional) $H$ for which there exists an invertible 
element ${\mf g}\in H$ satisfying (\ref{prop1}) and (\ref{pivstr2}). 
\end{definition}

Thus the definition of a sovereign quasi-Hopf algebra is designed in such a way that its category of finite-dimensional 
left representations is sovereign or, equivalently, pivotal.   

Under some conditions imposed to the field $k$, a family of sovereign quasi-Hopf algebras $H$ is given by the semisimple ones. 
Recall that $H$ is semisimple if it is semisimple as an algebra, and that this is equivalent to the existence 
of a left integral $t$ in $H$ (i.e. of an element $t\in H$ obeying $ht=\va(h)t$, for all $h\in H$) such that $\va(t)=1$. 
If this is the case then $t$ is also a right integral in $H$, that is $th=\va(h)t$, for all $h\in H$. We refer to 
\cite{pan} for more details related to this topic.     

\begin{example}\exlabel{ssissovereign}
A finite dimensional semisimple quasi-Hopf algebra over an algebraic closed field of characteristic zero is sovereign 
via the element ${\mf g}:=q^2t_2p^2S(q^1t_1p^1)$, where $p_R$, $q_R$ are the elements defined by
\begin{eqnarray*}
p_R&=&p^1\ot p^2=x^1\ot x^2\b S(x^3),\\
q_R&=&q^1\ot q^2=X^1\ot S^{-1}(\a X^3)X^2,
\end{eqnarray*}
and $t$ is a left (and so right as well) integral in $H$ satisfying $\va(t)=1$. 
\end{example}
\begin{proof}
It follows from \cite[Propositions 8.24 \& 8.23]{eno} that ${}_H{\cal M}^{\rm fd}$ has a unique 
pivotal ($\equiv$ sovereign) structure such that $\un{{\rm dim}}(V):={\rm tr}_r(\Id_V)$ equals ${\rm dim}_k(V)$, 
for any simple object of ${}_H{\cal M}^{\rm fd}$. Furthermore, by \cite{mn, sch} the element ${\mf g}\in H$ that 
produces this pivotal structure on ${}_H{\cal M}^{\rm fd}$ is just the one mentioned in the statement.  
\end{proof}

We refer to \cite{hnRMP, btRecon} for the explicit quasi-Hopf algebra structure of the quantum double $D(H)$ of a finite dimensional 
quasi-Hopf algebra $H$. 

\begin{example}
If $H$ is a finite-dimensional sovereign quasi-Hopf algebra then so is its quantum double $D(H)$.   
\end{example}
\begin{proof}
The category ${}_{D(H)}{\cal M}^{\rm fd}$ identifies to the right centre of the monoidal category 
${}_H{\cal M}^{\rm fd}$, see \cite{bcp, bpvo}. Thus our result follows from \cite[Exercise 7.13.6]{egnobook}. 
\end{proof}

\begin{remark}\relabel{sovereinstrleftmodqHa}
If $H$ is a sovereign quasi-Hopf algebra via an invertible element ${\mf g}$ of $H$ obeying 
(\ref{prop1}) and (\ref{pivstr2}), then ${}_H{\cal M}^{\rm fd}$ is sovereign with the following 
rigid structure.

If $V\in {}_H{\cal M}^{\rm fd}$, with $H$-action denoted by $H\ot V\ni h\ot v\ra h\cdot v\in V$, 
then the left dual of $V$ is $V^*$ with $H$-module structure $(h\cdot v^*)(v)=v^*(S(h)\cdot v)$, for 
all $v^*\in V^*$, $v\in V$ and $h\in H$,   
and ${\rm ev}_V$ and ${\rm coev}_V$ given by 
\begin{eqnarray*}
&&{\rm ev}_V: V^*\ot V\ni v^*\ot v\mapsto v^*(\alpha\cdot v)\in k,\\
&&{\rm coev}_V: k\ni 1_k\mapsto \beta\cdot v_i\ot v^i\in V\ot V^*,
\end{eqnarray*}
where $\{v_i, v^i\}_i$ are dual bases in $V$ and $V^*$ (summation implicitly understood). 
The right 
dual of $V$ is again $V^*$ considered as a left $H$-module via the antipode $S$ of $H$ as above, 
and with the evaluation and coevaluation morphisms given by 
\begin{eqnarray*}
&&{\rm ev}'_V: V\ot V^*\ni v\ot v^*\mapsto v^*({\mf g}^{-1}S^{-1}(\a)\cdot v)=v^*(S(\a){\mf g}^{-1}\cdot v)\in k,\\
&&{\rm coev}'_V: k\ni 1_k\mapsto v^i\ot S^{-1}(\b){\mf g}\cdot v_i=v^i\ot {\mf g}S(\b)\cdot v_i\in V^*\ot V,
\end{eqnarray*}
summation implicitly understood, where $\{v_i, v^i\}_i$ are dual bases in $V$ and $V^*$. 
\end{remark} 

When we refer to a sovereign quasi-Hopf algebra we always indicate
the element that produces the sovereign structure, by pointing out the couple $(H, {\mf g})$. Also, for $V$ a left $H$-module 
we denote by ${\rm End}_H(V)$ the set of $H$-endomorphisms of $V$. 

\begin{corollary}
If $(H, {\mf g})$ is a sovereign quasi-Hopf algebra and $V$ is a finite dimensional left $H$-module then 
for all $f\in {\rm End}_H(V)$ we have 
\[
{\rm tr}_l(f)={\rm tr}(V\ni v\mapsto f({\mf g}S(\b)\a\cdot v)\in V)~~\mbox{and}~~
{\rm tr}_r(f)={\rm tr}(V\ni v\mapsto f({\mf g}^{-1}\b S(\a)\cdot v)\in V),
\]
where by ${\rm tr}(\chi)$ we denote the usual trace of a $k$-linear endomorphism $\chi$. 
\end{corollary}
\begin{proof}
By \reref{sovereinstrleftmodqHa} and \equref{defoflrtraces}, for all $f\in {\rm End}_H(V)$ we have 
\begin{eqnarray*}
{\rm tr}_l(f)&=&\le v^i, \a\cdot f({\mf g}S(\b)\cdot v_i)\ri\\
&=&\le S^{-1}(\a)\cdot v^i, f({\mf g}S(\b)\cdot v_i)\ri\\
&=&\le v^j, f({\mf g}S(\b)\a\cdot v_j)\ri
={\rm tr}(V\ni v\mapsto f({\mf g}S(\b)\a\cdot v)\in V),
\end{eqnarray*}
as stated. The formula for ${\rm tr}_r(f)$ can be computed in a similar manner, so we are done. 
\end{proof}

By analogy with the Hopf case \cite{barret1}, we call a quasi-Hopf algebra spherical if its category of left 
finite dimensional representations is spherical. By the above results we have the following.

\begin{proposition}
A quasi-Hopf algebra is spherical if and only if there exists an invertible element 
${\mf g}\in H$ satisfying (\ref{prop1}) and (\ref{pivstr2}), and such that 
\begin{equation}\eqlabel{sphericalcondquasi}
{\rm tr}(V\ni v\mapsto f({\mf g}S(\b)\a\cdot v)\in V)=
{\rm tr}(V\ni v\mapsto f({\mf g}^{-1}\b S(\a)\cdot v)\in V),
\end{equation}
for any finite dimensional left $H$-module $V$ and any $f\in {\rm End}_H(V)$. 
\end{proposition}

As any ribbon category is spherical we get the following. 

\begin{example}
If $(H, R, \eta)$ is a ribbon quasi-Hopf algebra as in \deref{defribbQHA} then ${}_H{\cal M}^{\rm fd}$ is spherical 
via the element ${\mf g}=(u\eta)^{-1}$, where $u$ is as in \equref{elmu}.
\end{example}

An important class of spherical categories is defined by involutory Hopf algebras, which coincides 
to the class of semsimple Hopf algebras provided that $k$ has characteristic zero (see \cite{lr1}). 
As we will see a similar result holds for quasi-Hopf algebras. 

Recall from \cite{bctInv} that a quasi-Hopf algebra is called involutory if 
\begin{equation}\eqlabel{invqHa}   
S^2(h)=S(\b)\a h\b S(\a),~\forall~h\in H.
\end{equation}
If $H$ is involutory then $\a, \b\in H$ are invertible and $(\b S(\a))^{-1}=S(\b)\a$, 
cf. \cite[Lemma 3.2]{bctInv}. 

\begin{theorem}\thlabel{invisspherical}
Any involutory quasi-Hopf algebra is spherical, and therefore sovereign, too.
\end{theorem}
\begin{proof}
We show that ${\mf g}:=\b S(\a)$ has all the required properties. In fact, by the above comments we have 
$S^2(h)={\mf g}^{-1}h{\mf g}$, for all $h\in H$, and the equality in \equref{sphericalcondquasi} 
is fulfilled since the both sides of it are equal to the usual trace of $f$. So it remains to show that 
${\mf g}=\b S(\a)$ satisfies (\ref{pivstr2}). 
If we assume $k$ algebraic closed of characteristic zero then this follows from 
\cite[Remarks 3.5 1)]{bctInv}. As we work over an arbitrary field, a proof for the fact that 
${\mf g}=\b S(\a)$ satisfies (\ref{pivstr2}) should be given. 
Towards this end, for the quasi-Hopf algebra $(H^{\rm cop}, \Delta^{\rm cop}, \va, (\Phi^{-1})^{321}:=x^3\ot x^2\ot x^1, 
S^{-1}, S^{-1}(\a), S^{-1}(\b))$ and the tensor product algebra $A=H\ot H$ consider $f, g, \r, \sigma$ as in \leref{prelantcoalgqHa}, i.e. 
\[
f=\Delta^{\rm cop},~g=\Delta^{\rm cop}\circ S^{-1},~\r=\Delta(S^{-1}(\a)),~\sigma=\Delta^{\rm cop}(S^{-1}(\b)).
\]
Define the anti-algebra map $\ov{g}=\Delta^{\rm cop}\circ S: H^{\rm cop}\ra H\ot H$ and the elements 
$\ov{\rho}=\Delta^{\rm cop}(\b^{-1})$, $\ov{\sigma}=\Delta^{\rm cop}(\a^{-1})\in H$. By \cite[Proposition 3.4]{bctInv} 
we have $S(h_2)\b^{-1}h_1=\va(h)\b^{-1}$ and $h_2\a^{-1}S(h_1)=\va(h)\a^{-1}$, for all $h\in H$, and therefore 
$(f, \ov{g}, \ov{\rho}, \ov{\sigma})$ satisfies for $H^{\rm cop}$ and $A=H\ot H$ the relations in \equref{qHaf5}. 
The ones in \equref{qHaf6} are satisfied as well since, 
according to \cite[Lemma 3.2]{bctInv}, $\b S(\b\a)=\a^{-1}$ and $S(\b\a)\a=\b^{-1}$. Therefore  
\begin{eqnarray*}
x^3\a^{-1}S(x^2)\b^{-1}x^1&=&
x^3\b S(\a)S(\a x^2\b)S(\b)\a x^1\\
&\equal{\equref{invqHa}}&
x^3S^{-1}(\a x^2\b )x^1=S^{-1}(S(x^1)\a x^2\b S(x^3))=1,
\end{eqnarray*}
and similarly
\begin{eqnarray*}
S(X^3)\b^{-1}X^2\a^{-1}S(X^1)&=&S(\a X^3)S(\b)\a X^2\b S(\a)S(X^1\b)\\
&\equal{\equref{invqHa}}&S(\a X^3)S^2(X^2)S(X^1\b)=S(X^1\b S(X^2)\a X^3)=1.
\end{eqnarray*}
Hence, \leref{prelantcoalgqHa} guarantees the existence and the uniqueness of an invertible element $F\in H\ot H$ 
satisfying, for all $h\in H$, the relations
\[
\Delta^{\rm cop}(S(h))=F\Delta^{\rm cop}(S^{-1}(h))F^{-1},~\Delta^{\rm cop}(\a^{-1})=\Delta^{\rm cop}(S^{-1}(\b))F,~
\Delta^{\rm cop}(\b^{-1})=F\Delta^{\rm cop}(S^{-1}(\a)).
\]
Now, from \equref{invqHa} and \equref{ca} we see that 
\begin{eqnarray*}
&&\hspace*{-2cm}
f_{21}^{-1}(S\ot S)(f)({\mf g}^{-1}\ot {\mf g}^{-1})\Delta^{\rm cop}(S^{-1}(h))({\mf g}\ot {\mf g})(S\ot S)(f^{-1})f_{21}\\
&=&f_{21}^{-1}({\mf g}^{-1}\ot {\mf g}^{-1})(S^{-1}\ot S^{-1})(f)\Delta^{\rm cop}(S^{-1}(h))(S^{-1}\ot S^{-1})(f^{-1})({\mf g}\ot {\mf g})f_{21}\\
&=&f_{21}^{-1}({\mf g}^{-1}\ot {\mf g}^{-1})(S^{-1}\ot S^{-1})(\Delta(h))({\mf g}\ot {\mf g})f_{21}\\
&=&f_{21}^{-1}(S\ot S)(\Delta(h))f_{21}=\Delta^{\rm cop}(S(h)),
\end{eqnarray*}  
for all $h\in H$. We get from here that $F=f_{21}^{-1}(S\ot S)(f)({\mf g}^{-1}\ot {\mf g}^{-1})$, which together with 
$\Delta^{\rm cop}(\a^{-1})=\Delta^{\rm cop}(S^{-1}(\b))F^{-1}$ and $\a^{-1}=S^{-1}(\a\b)\b$ leads to 
\[
\Delta(S^{-1}(\a)\b)=({\mf g}\ot {\mf g})(S\ot S)(f_{21}^{-1})f.
\]
But $S(\b S(\a))=S^2(\a)S(\b)=S(\b)\a^2\b S(\b\a)=S(\b)\a$ because, once more, $\a^{-1}=\b S(\b\a)$. Thus  
${\mf g}=\b S(\a)=S^{-1}(S(\b)\a)=S^{-1}(\a)\b$, and this finishes the proof.    
\end{proof}

It was proved in \cite[Proposition 4.4]{bctInv} that $D(H)$ is involutory if $H$ is so, provided that the relation 
\begin{equation}\eqlabel{impformulaforinvQDqHa}
\Delta(S(\b)\a)=f^{-1}(S\ot S)(f_{21})(S(\b)\a\ot S(\b)\a)
\end{equation}
holds. As ${\mf g}^{-1}=S(\b)\a$ it follows from the proof of \thref{invisspherical} that \equref{impformulaforinvQDqHa} is valid for any involutory quasi-Hopf algebra. 
So we get the following result, which in particular produces a new class of spherical quasi-Hopf algebras. It 
answer also in positive to a question raised in \cite{bctInv}. 

\begin{corollary}\colabel{invQDqHa}
A quasi-Hopf algebra $H$ is involutory if and only if so is its quantum double $D(H)$.
\end{corollary}
\begin{proof}
The direct implication follows from the above comments, while the converse follows since $H$ is a quasi-Hopf subalgebra of $D(H)$. 
\end{proof}

Hence all the examples of involutory quasi-Hopf algebras presented in \cite{bctInv} are as well examples of 
spherical (and implicitly sovereign) quasi-Hopf algebras. Another one is the following.

\begin{example}\exlabel{Cnissovereign}
Let $k$ be a field containing a primitive root of unity $q$ of degree $n$ (so in particular $n\not=0$ in $k$) and $C_n$ the cyclic group 
of order $n$, say generated by $g$. We endow the group algebra $k[C_n]$ of $k$ and $C_n$ with the comultiplication given by 
$\Delta(g)=g\ot g$ and $\va(g)=1$, for all $g\in G$, extended to the whole $k[C_n]$ as algebra morphisms. In this way we can see $k[C_n]$ as a 
quasi-bialgebra with reassociator $\Phi_q=\sum\limits_{i, j, l=0}^{n-1}q^{i\left[\frac{j+l}{n}\right]}1_i\ot 1_j\ot 1_l$, where 
$[a]$ stands for the integer part of the rational number $a$ and $1_i:=\frac{1}{n}\sum\limits_{a=0}^{n-1}q^{(n-i)a}g^a$, for all $1\leq i\leq n-1$. 
Furthermore, this structure of $k[C_n]$ can be completed up to a quasi-Hopf algebra one by defining $S(g)=g^{-1}$, extended to the whole 
$k[C_n]$ as an anti-morphisms of algebras, $\a=g^{-1}$ and $\b=1$. We denote by $H_q(n)$ this quasi-Hopf algebra structure of $k[C_n]$.  

$H_q(n)$ is an involutory quasi-Hopf algebra, and so $(H_q(n), g)$ is a spherical quasi-Hopf algebra. 
\end{example}
\begin{proof}
That $H_q(n)$ is a quasi-Hopf algebra is a well known fact, see for instance \cite{gel}. Actually, since $H_q(n)$ is commutative, the only thing we must check is that $\Phi_q$ is a normalized $3$-cocycle, that is $\Phi_q$ satisfies the relation \equref{q1}. This follows 
from the fact that $\Phi_q$ is the $3$-cocycle (in the Harrison cohomology, see \cite{bctKlein}) dual to the $3$-cocycle of the group 
$C_n$ corresponding to $q$. 

We have $S^2=\Id$ and because $H_n(q)$ is commutative we can also see $S^2(h)={\mf g}^{-1}h{\mf g}$, for all $h\in H_q(n)$, 
with ${\mf g}=\b S(\a)=g$. Thus $H_n(q)$ is involutory, and so \thref{invisspherical} applies.         
\end{proof}

It was proved in \cite[Theorem 7.2]{mn} that for a finite dimensional semisimple 
quasi-Hopf algebra $H$ over an algebraic closed field of characteristic zero one has ${\mf g}^{-1}=S({\mf g})$, 
where ${\mf g}$ is the element of $H$ responsible for the sovereign structure of $H$ as in \exref{ssissovereign}. 
As we will see, this is a consequence of a more general result valid for sovereign quasi-Hopf algebras 
(over arbitrary fields). 

\begin{proposition}
If $(H, {\mf g})$ is a sovereign quasi-Hopf algebra then ${\mf g}^{-1}=S({\mf g})$. 
\end{proposition}
\begin{proof}
As ${\mf g}$ is invertible, it suffices to show that ${\mf g}S({\mf g})=1$. Indeed, by applying 
$\va\ot \va$ to the both sides of (\ref{pivstr2}) we get $\va({\mf g})=1$. Thus, by using again 
(\ref{pivstr2}) we see that 
\begin{eqnarray*}
\b&=&\va({\mf g})\b={\mf g}_1\b S({\mf g}_2)\\
&=&{\mf g}S(g^2)f^1\b S({\mf g}S(g^1)f^2)
\equal{(\ref{l3})}{\mf g}S(g^2)S(\a)S({\mf g}S(g^1))\\
&\equal{(\ref{l3})}&{\mf g}S^2(\b)S({\mf g})\equal{(\ref{prop1})}\b {\mf g}S({\mf g}).
\end{eqnarray*}
On the other hand, (\ref{prop1}) allows to compute that  
\[
A\b B=A\b {\mf g}S({\mf g})B=A\b {\mf g}S(S^{-1}(B){\mf g})=A\b {\mf g}S({\mf g}S(B))=A\b {\mf g}S^2(B)S({\mf g})
=A\b B{\mf g}S({\mf g}),
\]  
for all $A\ot B\in H\ot H$. By taking $A\ot B=X^1\ot S(X^2)\a X^3$ in the above equality, by \equref{q6} we conclude that 
$1={\mf g}S({\mf g})$, as needed.
\end{proof} 

\section{Balanced categories and balanced quasi-bialgebras}\selabel{balancedcatqHa}
\setcounter{equation}{0}
Kassel and Turaev \cite{kasturaev} associate to any rigid monoidal category $\Cc$ a ribbon category $\Dc(\Cc)$ 
and specialize this construction to the category of representations of a Hopf algebra. 
The importance of their construction resides in the fact that any ribbon category produces link invariants (see \cite{turaev}), 
and so out of any rigid monoidal category $\Cc$ we can construct link invariants taking values in the semigroup of the endomorphisms of the 
unit object $\un{1}$ of $\Cc$.  

In this section we present a more general construction, in the sense that in place of the 
centre we consider an arbitrary braided category. Our construction has the advantage that in the quasi-Hopf case it leads to a more conceptual and less 
computational proof for certain balanced/ribbon isomorphisms of categories. 
\subsection{Balanced categories obtained from braided categories}
Following the idea of Kassel and Turaev \cite{kasturaev}, balanced
categories from monoidal ones were obtained by Drabant in \cite{drabant}. The construction of Drabant generalizes as follows.  

\begin{proposition}\prlabel{balancedfrombraided}
Let $(\Cc, c)$ be a braided category. Denote by ${\cal B}(\Cc, c)$ the category whose

$\bullet$ objects are pairs $(V, \eta_V)$ consisting of an object $V\in \Cc$ and an automorphism 
$\eta_V: V\ra V$ of $V$ in $\Cc$;

$\bullet$ morphisms between $(V, \eta_V)$ and $(W, \eta_W)$ are morphisms 
$f: V\ra W$ in $\Cc$ fulfilling $\eta_Wf=f\eta_V$. 

Then ${\cal B}(\Cc, c)$ is a balanced category via the following structure:

(i) The tensor product on ${\cal B}(\Cc, c)$ is given by $(V, \eta_V)\ot (W, \eta_W)=(V\ot W, \eta_{V\ot W})$, with
\[
\eta_{V\ot W}=(\eta_V\ot \eta_W)c_{W, V}c_{V, W}, 
\]
and on morphisms it acts as the tensor product of $\Cc$. 
The unit object in ${\cal B}(\Cc, c)$ is $(\un{1}, \eta_{\un{1}}:=\Id_{\un{1}})$.  
Together with the associativity and the left and right unit constraints of $\Cc$ these give the 
monoidal structure on ${\cal B}(\Cc, c)$; 

(ii) The braiding on ${\cal B}(\Cc, c)$ is determined by the braiding $c$ of $\Cc$;

(iii) The balancing on ${\cal B}(\Cc, c)$ is produced by $\eta:=(\eta_V)_{(V, \eta_V)\in {\cal B}(\Cc, c)}$.  
\end{proposition}
\begin{proof}
We check that the associativity constraint of $\Cc$ is a morphism in ${\cal B}(\Cc, c)$; 
the remaining details are trivial. Assuming $\Cc$ strict monoidal, this reduces to 
the equality $\eta_{(V\ot W)\ot T}=\eta_{V\ot (W\ot T)}$, for all $V, W, T\in \Cc$. In diagrammatic notation, 
the latter comes out as 
\[{{\footnotesize
\gbeg{3}{9}
\got{1}{V}\got{1}{W}\got{1}{T}\gnl
\gcl{1}\gbr\gnl
\gbr\gcl{1}\gnl
\gbr\gcl{1}\gnl
\gcl{1}\gbr\gnl
\gbr\gcl{1}\gnl
\gbr\gmp{\eta_T}\gnl
\gmp{\eta_V}\gmp{\eta_W}\gcl{1}\gnl
\gob{1}{V}\gob{1}{W}\gob{1}{T}
\gend
}}
=
{{\footnotesize
\gbeg{3}{9}
\got{1}{V}\got{1}{W}\got{1}{T}\gnl
\gbr\gcl{1}\gnl
\gcl{1}\gbr\gnl
\gcl{1}\gbr\gnl
\gbr\gcl{1}\gnl
\gcl{1}\gbr\gnl
\gmp{\eta_V}\gbr\gnl
\gcl{1}\gmp{\eta_W}\gmp{\eta_T}\gnl
\gob{1}{V}\gob{1}{W}\gob{1}{T}
\gend
}}~.
\]
To prove the above equality we compute the right hand side of it as follows: apply twice 
\equref{qybe}, and then apply the naturality of the braiding 
$c$ to the morphism $c_{W, V}c_{V. W}: V\ot W\ra V\ot W$. In this way we get the 
left hand side of the equality. 
\end{proof}

\begin{remark}\relabel{balancedFromQuasiHopf}
By \prref{balancedfrombraided} we can associate to any monoidal category $\Cc$ two balanced ones. 
Namely, ${\mf B}_l(\Cc):={\cal B}({\cal Z}_l(\Cc), c)$ and ${\mf B}_r(\Cc):={\cal B}({\cal Z}_r(\Cc), c)$, where 
${\cal Z}_{l/r}(\Cc)$ is the left/right center of $\Cc$. 

As a concrete example, we can take $\Cc={}_H{\cal M}$, $H$ a quasi-Hopf algebra with bijective antipode, 
in which case we have ${\mf B}_l(\Cc)={\cal B}({}_H^H\YD, c)$. Here ${}_H^H\YD$ is the category of left-right 
Yetter-Drinfeld modules over $H$, a braided category with braiding $c$ given by $c_{V, W}: V\ot W\ni v\ot w\mapsto v_{(-1)}\cdot w\ot v_{(0)}\in W\ot V$, for all $V, W\in {}_H^H\YD$, where $V\ni v\mapsto v_{(-1)}\ot v_{(0)}\in H\ot H$ is our notation for 
the right $H$-action on $V$ (more details can be found in \cite{bcp}). Analogously,  ${\mf B}_r(\Cc)={\cal B}({}_H\YD^H, \mf{c})$, where  
${}_H\YD^H$ is the category of left-right Yetter-Drinfeld modules over $H$ endowed with the braiding 
$\mf{c}$ defined by 
\begin{equation}\eqlabel{lryd4}
{\mf c}_{V, W}: V\ot W\ni v\ot w\mapsto w_{(0)}\ot w_{(1)}\cdot v\in W\ot V, 
\end{equation}
for all $V, W\in {}_H\YD^H$ (this time $W\ni w\mapsto w_{(0)}\ot w_{(1)}\in W\ot H$ is the notation for 
the right coaction of $H$ on $W$). 
\end{remark}

\prref{balancedfrombraided} can be applied also to the category ${}_H{\cal M}$, where $(H, R)$ is a 
QT quasi-Hopf algebra. We will exploit this fact in what follows.    

\subsection{A class of balanced quasi-Hopf algebras}\sselabel{balancedquasiHopf}
To a finite dimensional quasi-Hopf algebra $H$ we can associate a QT one, its 
quantum double $D(H)$. Furthermore, the category of $D(H)$-representations is braided isomorphic to the category of Yetter-Drinfeld modules. 
In this subsection we will go further by showing that to any QT quasi-bialgebra $(H, R)$ we can associate a balanced one, denoted by $H[\theta, \theta^{-1}]$, 
and that the category of $H[\theta, \theta^{-1}]$-representations is balanced isomorphic to ${\cal B}({}_H{\cal M}, c)$, where 
$c$ is the braiding of ${}_H{\cal M}$ defined by the $R$-matrix $R$. . 

More exactly, for $H$ a quasi-bialgebra we denote by 
$H[\theta, \theta^{-1}]$ the free $k$-algebra generated by $H$ and $\theta$, with relations 
$h\theta=\theta h$, for all $h\in H$, and $\theta\theta^{-1}=\theta^{-1}\theta=1$; 
by analogy with the commutative case and the terminology used in \cite{drabant}, we call 
$H[\theta, \theta^{-1}]$ the Laurent polynomial algebra over $H$.  

\begin{proposition}\prlabel{ribbquasibialgfromQTones}
Let $(H, R)$ be a QT quasi-bialgebra. Then $H[\theta, \theta^{-1}]$ 
is a balanced quasi-bialgebra with structure given by 
\[
\Delta\mid_H=\Delta_H~,~\Delta(\theta^{\pm})=(R_{21}R)^{\mp}(\theta^{\pm}\ot \theta^{\pm})~,~
\va\mid_H=\va_H~,~\va(\theta^{\pm})=1,   
\]
where, for simplicity, we denote $\theta=\theta^{+}$ and $\theta^{-}=\theta^{-1}$, and similarly for $(R_{21}R)^\pm$.
\end{proposition}
\begin{proof}
The only thing we have to prove is the quasi-coassociativity of $\Delta$ on $\theta$. It will follow then 
that the natural inclusion of $H$ into $H[\theta, \theta^{-1}]$ is a quasi-bialgebra morphism, and that 
$H[\theta, \theta^{-1}]$ is balanced via $\theta$.   

Now, since $\Delta(h)R_{21}R=R_{21}R\Delta(h)$, for all $h\in H$, the quasi-coassociativity of $\Delta$ on 
$\theta$ reduces to  
\[
\Phi(R_{21}R\ot 1_H)(\Delta_H\ot \Id_H)(R_{21}R)\Phi^{-1}=
(1_H\ot R_{21}R)(\Id_H\ot \Delta_H)(R_{21}R).
\]
This can be restated as 
\[
\Phi(\Delta_H\ot \Id_H)(R_{21}R)R_{21}R_{12}\Phi^{-1}=
R_{32}R_{23}(\Id_H\ot \Delta_H)(R_{21}R).
\]
By using \equref{qt1} and \equref{qt2}, the left hand side of the above equality equals 
\[
\Phi(\Delta(R^2)\ot R^1)(\Delta(r^1)\ot r^2)R_{21}R_{12}\Phi^{-1}=
R_{32}\Phi_{132}R_{31}R_{13}\Phi^{-1}_{132}R_{23}\Phi R_{21}R_{12}\Phi^{-1},
\]
while its right hand side is equal to 
\[
R_{32}R_{23}(R^2\ot \Delta_H(R^1))(r^1\ot \Delta_H(r^2))
=R_{32}R_{23}\Phi R_{21}\Phi^{-1}_{213}R_{31}R_{13}\Phi_{213}R_{12}\Phi^{-1}.
\]
Thus we must show that 
\[
\Phi_{132}R_{31}R_{13}\Phi^{-1}_{132}R_{23}\Phi R_{21}=
R_{23}\Phi R_{21}\Phi^{-1}_{213}R_{31}R_{13}\Phi_{213}.
\]
If $\tau$ is the usual switch for the category of $k$-vector spaces, the latter follows from 
\begin{eqnarray*}
&&\hspace*{-1.5cm}
R_{23}\Phi R_{21}\Phi^{-1}_{213}R_{31}R_{13}\Phi_{213}\\
&=&
(\tau\ot \Id_H)(R_{13}\Phi_{213}R_{12}\Phi^{-1})(\tau\ot \Id_H)(1_H\ot R_{21}R)\Phi_{213}\\
&\equal{\equref{qt2}}&
(\tau\ot \Id_H)(\Phi_{231}(\Id_H\ot \Delta_H)(R)(1_H\ot R_{21}R))\Phi_{231}\\
&=&\Phi_{132}(\tau\ot \Id_H)(1_H\ot R_{21}R)(\tau\ot \Id_H)(\Id_H\ot \Delta(R))\Phi_{213}\\
&\equal{\equref{qt2}}&
\Phi_{132}R_{31}R_{13}(\tau\ot \Id_H)(\Phi^{-1}_{231}R_{13}\Phi_{213}R_{12}\Phi^{-1})\Phi_{213}\\
&=&\Phi_{132}R_{31}R_{13}\Phi^{-1}_{132}R_{23}\Phi R_{21}, 
\end{eqnarray*}
as desired. 
\end{proof}

\begin{proposition}\prlabel{balancedqHacase}
Let $(H, R)$ be a QT quasi-bialgebra and $\Cc={}_H{\cal M}$. Then ${\cal B}(\Cc, c)$ and 
${}_{H[\theta, \theta^{-1}]}{\cal M}$ are isomorphic as balanced categories, where 
$c$ is as in \equref{braiQT} and ${\cal B}(\Cc, c)$ is as in the last part of \reref{balancedFromQuasiHopf}. 
\end{proposition}
\begin{proof}
Indeed, the desired isomorphism is produced by the following correspondence. To $(V, \eta_V)$ 
in ${\cal B}(\Cc, c)$ we associate $V$ regarded as a left $H[\theta, \theta^{-1}]$-module via 
the $H$-module structure of it and $\theta^\pm\cdot v=\eta_V^\mp(v)$, for all $v\in V$. In this way a morphism 
in ${\cal B}(\Cc, c)$ becomes a morphism in ${}_{H[\theta, \theta^{-1}]}{\cal M}$.     
\end{proof}

\section{Ribbon categories and ribbon quasi-Hopf algebras}\selabel{ribboncategqHa}
\setcounter{equation}{0}
To any left rigid braided category $(\Cc, c)$ (which is consequently rigid) we assign a 
ribbon category ${\cal R}(\Cc, c)$. Thus, to any left (resp. right) rigid monoidal category $\Cc$ we 
can associate a ribbon monoidal one, that will be denoted by ${\mf R}_l(\Cc)$ 
(resp. ${\mf R}_r(\Cc)$). The latter are possible due to the left and right center constructions and coincide to the 
ones defined in \cite{kasturaev}. Our construction allows to associate to any QT quasi-Hopf algebra a ribbon one, and suggests also 
how to construct a class of ribbon quasi-Hopf algebras. 
\subsection{Ribbon categories obtained from rigid monoidal categories}\sselabel{ribbfrombraidandmonoid}
Inspired by the formula in \equref{msquaretwist} we introduce the following category, that will turn out to 
be a ribbon category. 

\begin{definition}\delabel{ribbfrombraidwithld}
If $(\Cc, c)$ is a left rigid braided (strict) monoidal category 
then ${\cal R}(\Cc, c)$ is the category whose 

$\bullet$ objects are pairs $(V, \eta_V)$ consisting of an object $V$ of $\Cc$ and an 
automorphism $\eta_V$ of $V$ in $\Cc$ satisfying
\begin{equation}\eqlabel{etamorgen}
\eta_V^{-2}=
{{\footnotesize
\gbeg{3}{6}
\gvac{2}\got{1}{V}\gnl
\gdb\gcl{3}\gnl
\gbr\gnl
\gbr\gnl
\gcl{1}\gev\gnl
\gob{1}{V}
\gend
}}
=
{{\footnotesize
\gbeg{3}{6}
\gvac{2}\got{1}{V}\gnl
\gdb\gcl{2}\gnl
\gbr\gnl
\gcl{1}\gibr\gnl
\gev\gcl{1}\gnl
\gvac{2}\gob{1}{V}
\gend
}}~;
\end{equation}

$\bullet$ morphisms $f: (V, \eta_V)\ra (W, \eta_W)$ are morphisms $f: V\ra W$ in $\Cc$ such that 
$\eta_Wf=f\eta_V$. 

The composition in ${\cal R}(\Cc)$ is given by the composition in $\Cc$, and 
the identity morphism of an object $(V, \eta_V)$ is $\Id_V$. 
\end{definition}

Otherwise stated, ${\cal R}(\Cc, c)$ is the full subcategory of ${\cal B}(\Cc, c)$ 
considered in \prref{balancedfrombraided} determined by those objects 
$(V, \eta_V)$ of ${\cal B}(\Cc, c)$ for which $\eta_V$ obeys \equref{etamorgen}. As we pointed out in 
the second part of \prref{balancedfaarfromribbon}, this is the necessary and sufficient condition that 
turns ${\cal B}(\Cc, c)$ into a ribbon category. Actually, the ribbon structure of ${\cal R}(\Cc, c)$ is encoded in 
the following result. 

\begin{theorem}\thlabel{ribbonfrombraided}
Let $(\Cc, c)$ be a left rigid braided (strict) monoidal category. Then ${\cal R}(\Cc, c)$ 
is a ribbon monoidal category as follows:  

$\bullet$ If $(V, \eta_V)$, $(W, \eta_W)\in {\cal R}(\Cc, c)$ 
then $(V, \eta_V)\ot (W, \eta_W)=(V\ot W, \eta_{V\ot W})$, where 
\begin{equation}
\eta_{V\ot W}=(\eta_V\ot\eta_W)c_{W,V}c_{V,W};\eqlabel{etamonoidalbr}
\end{equation} 

$\bullet$ The unit object is $(\un{1}, \eta_{\un{1}}=\Id_{\un{1}})$, and the 
associativity and the left and right unit constraints are the same as those of $\Cc$;

$\bullet$ The braiding equals $c$, regarded as an isomorphism in ${\cal R}(\Cc, c)$; 

$\bullet$ For $(V, \eta_V)$ an object in ${\cal R}(\Cc, c)$, a left dual object for it is $(V^*, \eta_{V^*})$, where
\begin{equation}
\eta_{V^*}=(\eta_V)^*,\eqlabel{twistdualobjribbbr}
\end{equation}
with evaluation and coevaluation morphisms equal to ${\rm ev}_V$ and ${\rm coev}_V$, viewed now as morphisms in ${\cal R}(\Cc, c)$; 

$\bullet$ The twist is given by 
\begin{equation}
\eta_V: (V, \eta_V)\ra (V, \eta_V),\eqlabel{twistribbonconstrbr}
\end{equation} 
an automorphism in ${\cal R}(\Cc, c)$. 
\end{theorem}
\begin{proof}
We start by proving that $(V\ot W, \eta_{V\ot W})$ is an object of ${\cal R}(\Cc, c)$, i.e. $\eta_{V\ot W}$ 
in \equref{etamonoidalbr} obeys \equref{etamorgen}. 
To this end, we need the equalities
\begin{equation}\eqlabel{natbraidribb1}
(a)~~{{\footnotesize
\gbeg{4}{6}
\got{1}{V}\got{1}{W}\got{1}{X}\got{1}{Y}\gnl
\gcl{1}\gcl{1}\gcl{1}\gcl{1}\gnl
\gcl{1}\gcl{1}\gsbox{2}\gnot{\hspace*{5mm}f}\gnl
\gcl{1}\gbr\gnl
\gbr\gcl{1}\gnl
\gob{1}{Z}\gob{1}{V}\gob{1}{W}
\gend}}=
{{\footnotesize
\gbeg{4}{7}
\got{1}{V}\got{1}{W}\got{1}{X}\got{1}{Y}\gnl
\gcl{1}\gbr\gcl{1}\gnl
\gbr\gbr\gnl
\gcl{1}\gbr\gcl{3}\gnl
\gsbox{2}\gnot{\hspace*{5mm}f}\gvac{2}\gcl{2}\gnl
\gcn{1}{1}{2}{2}\gnl
\gob{2}{Z}\gob{1}{V}\gob{1}{W}
\gend
}}~,~(b)~~
{{\footnotesize
\gbeg{4}{6}
\got{1}{V}\gvac{1}\got{1}{X}\gnl
\gcl{1}\gvac{1}\gcn{1}{1}{1}{1}\gnl
\gcl{1}\gsbox{3}\gnot{\hspace*{1cm}g}\gnl
\gbr\gcl{1}\gcl{1}\gnl
\gcl{1}\gbr\gcl{1}\gnl
\gob{1}{Y}\gob{1}{Z}\gob{1}{V}\gob{1}{T}
\gend
}}=
{{\footnotesize
\gbeg{4}{5}
\gvac{2}\got{1}{V}\got{1}{X}\gnl
\gvac{2}\gbr\gnl
\gsbox{3}\gnot{\hspace*{1cm}g}\gvac{3}\gcl{1}\gnl
\gcl{1}\gcl{1}\gibr\gnl
\gob{1}{Y}\gob{1}{Z}\gob{1}{V}\gob{1}{T}
\gend
}},
\end{equation}
for any morphisms $f: X\ot Y\ra Z$ and $g: X\ra Y\ot Z\ot T$ in $\Cc$, and respectively
\begin{equation}\eqlabel{natbraidribb2}
{{\footnotesize
\gbeg{3}{5}
\got{1}{X}\got{1}{Y}\got{1}{Z}\gnl
\gibr\gcl{1}\gnl
\gcl{1}\gibr\gnl
\gibr\gcl{1}\gnl
\gob{1}{Z}\gob{1}{Y}\gob{1}{X}
\gend
}}=
{{\footnotesize
\gbeg{3}{5}
\got{1}{X}\got{1}{Y}\got{1}{Z}\gnl
\gcl{1}\gibr\gnl
\gibr\gcl{1}\gnl
\gcl{1}\gibr\gnl
\gob{1}{Z}\gob{1}{Y}\gob{1}{X}
\gend}},
\end{equation}
valid for all $X, Y, Z\in \Cc$, which follow from the fact that $c_{-,-}$  is a natural isomorphism. 
Note that \equref{natbraidribb2} is nothing but an equivalent form of the categorical version of the 
Yang-Baxter equation \equref{qybe}.  

Now, if $\l_{V, W}: (V\ot W)^*\ra W^*\ot V^*$ is the isomorphism in $\Cc$ defined in \equref{firstisomdual} and 
$\l^{-1}_{V, W}$ is its inverse, then we compute
\begin{eqnarray*}
&&\hspace*{-2mm}
{{\footnotesize
\gbeg{5}{10}
\gvac{3}\got{1}{V}\got{1}{W}\gnl
\gwdb{3}\gcl{3}\gcl{4}\gnl
\gsbox{2}\gnot{\hspace*{5mm}V\ot W}\gvac{2}\gcl{1}\gnl
\gcl{1}\gbr\gnl
\gbr\gibr\gnl
\gcl{1}\gibr\gibr\gnl
\gcl{1}\gcl{1}\gibr\gcl{1}\gnl
\gcl{1}\gsbox{2}\gnot{\hspace*{5mm}V\ot W}\gvac{2}\gcl{2}\gcl{2}\gnl
\gwev{3}\gnl
\gvac{3}\gob{1}{V}\gob{1}{W}
\gend
}}=
{{\footnotesize
\gbeg{7}{10}
\gvac{5}\got{1}{V}\got{1}{W}\gnl
\gvac{1}\gwdb{4}\gcl{3}\gcl{3}\gnl
\gvac{1}\gcl{1}\gdb\gcl{1}\gnl
\gvac{1}\gcl{1}\gcl{1}\gsbox{2}\gnot{\hspace*{5mm}\l^{-1}_{V, W}}\gnl
\gvac{1}\gcl{1}\gbr\gcn{1}{1}{3}{1}\gcn{1}{2}{3}{1}\gnl
\gvac{1}\gbr\gibr\gnl
\gsbox{2}\gnot{\hspace*{5mm}\l_{V, W}}\gvac{2}\gibr\gibr\gnl
\gcl{1}\gev\gibr\gcl{2}\gnl
\gwev{4}\gcl{1}\gnl
\gvac{4}\gob{1}{V}\gob{1}{W}
\gend}}
\equal{\equuref{natbraidribb1}{a}}
{{\footnotesize
\gbeg{6}{10}
\gvac{4}\got{1}{V}\got{1}{W}\gnl
\gwdb{4}\gcl{4}\gcl{4}\gnl
\gcl{1}\gdb\gcl{1}\gnl
\gcl{1}\gbr\gcl{1}\gnl
\gbr\gbr\gnl
\gcl{1}\gbr\gibr\gcl{1}\gnl
\gcl{1}\gcl{1}\gibr\gibr\gnl
\gcl{1}\gev\gibr\gcl{1}\gnl
\gwev{4}\gcl{1}\gcl{1}\gnl
\gvac{4}\gob{1}{V}\gob{1}{W}
\gend}}\\
&&
\equal{\equuref{natbraidribb1}{b}}
{{\footnotesize
\gbeg{6}{10}
\gvac{4}\got{1}{V}\got{1}{W}\gnl
\gdb\gdb\gcl{1}\gcl{1}\gnl
\gbr\gbr\gcl{1}\gcl{1}\gnl
\gcl{1}\gbr\gcl{1}\gcl{1}\gcl{1}\gnl
\gcl{1}\gcl{1}\gibr\gcl{1}\gcl{1}\gnl
\gcl{1}\gcl{1}\gcl{1}\gibr\gcl{1}\gnl
\gcl{1}\gcl{1}\gibr\gibr\gnl
\gcl{1}\gev\gibr\gcl{1}\gnl
\gwev{4}\gcl{1}\gcl{1}\gnl
\gvac{4}\gob{1}{V}\gob{1}{W}
\gend
}}
\equal{\equref{natbraidribb2}}
{{\footnotesize
\gbeg{6}{10}
\gvac{4}\got{1}{V}\got{1}{W}\gnl
\gdb\gdb\gcl{1}\gcl{1}\gnl
\gbr\gbr\gcl{1}\gcl{1}\gnl
\gcl{1}\gbr\gibr\gcl{1}\gnl
\gcl{1}\gcl{1}\gibr\gcl{1}\gcl{1}\gnl
\gcl{1}\gcl{1}\gcl{1}\gibr\gcl{1}\gnl
\gcl{1}\gev\gcl{1}\gibr\gnl
\gcl{1}\gvac{2}\gibr\gcl{2}\gnl
\gwev{4}\gcl{1}\gnl
\gvac{4}\gob{1}{V}\gob{1}{W}
\gend}}
\equalupdown{\equuref{natbraidribb1}{b}}{\equref{natbraidribb2}}
{{\footnotesize
\gbeg{6}{9}
\gvac{2}\got{1}{V}\gvac{2}\got{1}{W}\gnl
\gdb\gcl{1}\gdb\gcl{1}\gnl
\gbr\gcl{1}\gbr\gcl{1}\gnl
\gcl{1}\gibr\gcl{1}\gibr\gnl
\gcl{1}\gcl{1}\gbr\gcl{1}\gcl{1}\gnl
\gcl{1}\gbr\gibr\gcl{1}\gnl
\gcl{1}\gev\gcl{1}\gibr\gnl
\gwev{4}\gcl{1}\gcl{1}\gnl
\gvac{4}\gob{1}{V}\gob{1}{W}
\gend
}}\\
&&
\equalupdown{(*)}{\equref{etamorgen}}
{{\footnotesize
\gbeg{4}{10}
\gvac{2}\got{1}{V}\got{1}{W}\gnl
\gdb\gcl{1}\gcl{1}\gnl
\gbr\gcl{1}\gcl{1}\gnl
\gcl{1}\gibr\gcl{1}\gnl
\gcl{1}\gmp{\eta^{-2}_V}\gcl{1}\gcl{1}\gnl
\gcl{1}\gibr\gcl{1}\gnl
\gcl{1}\gcl{1}\gibr\gnl
\gcl{1}\gibr\gcl{1}\gnl
\gev\gibr\gnl
\gvac{2}\gob{1}{V}\gob{1}{W}
\gend
}}
\equal{\equref{natbraidribb2}}
{{\footnotesize
\gbeg{4}{10}
\gvac{2}\got{1}{V}\got{1}{W}\gnl
\gdb\gcl{1}\gcl{1}\gnl
\gbr\gcl{1}\gcl{1}\gnl
\gcl{1}\gibr\gcl{1}\gnl
\gcl{1}\gmp{\eta^{-2}_V}\gcl{1}\gcl{1}\gnl
\gcl{1}\gcl{1}\gibr\gnl
\gcl{1}\gibr\gcl{1}\gnl
\gcl{1}\gcl{1}\gibr\gnl
\gev\gibr\gnl
\gvac{2}\gob{1}{V}\gob{1}{W}
\gend}}
\equal{\equref{natbraidribb2}}
{{\footnotesize
\gbeg{4}{9}
\gvac{2}\got{1}{V}\got{1}{W}\gnl
\gdb\gcl{1}\gcl{1}\gnl
\gbr\gmp{\eta_V^{-2}}\gcl{1}\gnl
\gcl{1}\gcl{1}\gibr\gnl
\gcl{1}\gibr\gcl{1}\gnl
\gev\gibr\gnl
\gvac{2}\gibr\gnl
\gvac{2}\gibr\gnl
\gvac{2}\gob{1}{V}\gob{1}{W}
\gend}}
\equal{\equref{etamorgen}}
{{\footnotesize
\gbeg{2}{8}
\got{1}{V}\got{1}{W}\gnl
\gmp{\eta^{-1}_V}\gmp{\eta_W^{-1}}\gnl
\gibr\gnl
\gibr\gnl
\gmp{\eta^{-1}_V}\gmp{\eta_W^{-1}}\gnl
\gibr\gnl
\gibr\gnl
\gob{1}{V}\gob{1}{W}
\gend}}
=\eta^{-2}_{V\ot W},
\end{eqnarray*}
where (*) is 
\[
{{\footnotesize
\gbeg{4}{6}
\got{1}{V}\got{1}{W}\got{2}{X}\gnl
\gcl{1}\gcl{1}\gcn{1}{1}{2}{2}\gnl
\gcl{1}\gcl{1}\gsbox{2}\gnot{\hspace*{5mm}h}\gnl
\gcl{1}\gbr\gcl{2}\gnl
\gbr\gcl{1}\gnl
\gob{1}{Y}\gob{1}{Z}\gob{1}{V}\gob{1}{W}
\gend}}
=
{{\footnotesize
\gbeg{4}{7}
\gvac{1}\got{1}{V}\got{1}{W}\got{1}{X}\gnl
\gvac{1}\gcl{1}\gbr\gnl
\gvac{1}\gbr\gcl{1}\gnl
\gsbox{2}\gnot{\hspace*{5mm}h}\gvac{2}\gcl{1}\gcl{2}\gnl
\gcl{1}\gibr\gnl
\gcl{1}\gcl{1}\gibr\gnl
\gob{1}{Y}\gob{1}{Z}\gob{1}{V}\gob{1}{W}
\gend
}}~~\mbox{for}~~
h=
{{\footnotesize
\gbeg{2}{4}
\got{2}{\un{1}}\gnl
\gdb\gnl
\gbr\gnl
\gob{1}{W^*}\gob{1}{W}
\gend}}~,~i.e.~
{{\footnotesize
\gbeg{4}{6}
\got{1}{V}\got{1}{W}\gnl
\gcl{1}\gcl{1}\gdb\gnl
\gcl{1}\gcl{1}\gbr\gnl
\gcl{1}\gbr\gcl{2}\gnl
\gbr\gcl{1}\gnl
\gob{1}{W^*}\gob{1}{V}\gob{1}{W}\gob{1}{W}
\gend}}
=
{{\footnotesize
\gbeg{4}{6}
\gvac{2}\got{1}{V}\got{1}{W}\gnl
\gdb\gcl{1}\gcl{1}\gnl
\gbr\gcl{1}\gcl{2}\gnl
\gcl{1}\gibr\gnl
\gcl{1}\gcl{1}\gibr\gnl
\gob{1}{W^*}\gob{1}{V}\gob{1}{W}\gob{1}{W}
\gend
}}.
\]
This ends the proof of the fact that the tensor product of ${\cal R}(\Cc, c)$ is well defined at the level of objects. 
It is easy to see that for any two morphisms $f, f'$ in ${\cal R}(\Cc, c)$ 
their tensor product $f\ot f'$ in $\Cc$ is actually a morphism in ${\cal R}(\Cc, c)$.  
Therefore, we have a tensor product functor on 
${\cal R}(\Cc, c)$ that together with the associativity and the left and right unit constraints of $\Cc$ defines on 
${\cal R}(\Cc, c)$ a monoidal structure; the unchecked details are left to the reader.

We next show that $c$ provides a braiding on ${\cal R}(\Cc, c)$.  
The only thing that must be verified is that $c_{V, W}$ is a 
morphism in ${\cal R}(\Cc, c)$, for any objects $(V, \eta_V)$,  $(W, \eta_W)$ of ${\cal R}(\Cc, c)$. This follows directly 
from the definitions and from the naturality of $c$.

Let $(V^*, \eta_{V^*})$ be as in \equref{twistdualobjribbbr}. 
By using the naturality of $c$ one sees that   
\begin{eqnarray}
&&\hspace*{-7mm}
{{\footnotesize
\gbeg{3}{7}
\got{1}{V}\gnl
\gmp{\eta_V}\gdb\gnl
\gcl{1}\gbr\gnl
\gbr\gmp{\eta_V}\gnl
\gcl{1}\gcl{1}\gcl{1}\gnl
\gev\gcl{1}\gnl
\gvac{2}\gob{1}{V}
\gend
}}
=
{{\footnotesize
\gbeg{3}{7}
\got{1}{V}\gnl
\gcl{1}\gdb\gnl
\gcl{1}\gmp{\eta_V}\gcl{1}\gnl
\gmp{\eta_V}\gbr\gnl
\gbr\gcl{1}\gnl
\gev\gcl{1}\gnl
\gvac{2}\gob{1}{V}
\gend
}}
=
{{\footnotesize
\gbeg{3}{7}
\gvac{2}\got{1}{V}\gnl
\gdb\gcl{1}\gnl
\gmp{\eta_V}\gcl{1}\gmp{\eta_V}\gnl
\gbr\gcl{1}\gnl
\gcl{1}\gibr\gnl
\gev\gcl{1}\gnl
\gvac{2}\gob{1}{V}
\gend
}}
=
{{\footnotesize
\gbeg{3}{6}
\gvac{2}\got{1}{V}\gnl
\gdb\gmp{\eta_V}\gnl
\gbr\gcl{1}\gnl
\gcl{1}\gibr\gnl
\gev\gmp{\eta_V}\gnl
\gvac{2}\gob{1}{V}
\gend
}}
=
\Id_V,\eqlabel{rightdualribbon}\\
&&\hspace*{-7mm}
\eta^{-2}_V=
{{\footnotesize
\gbeg{3}{6}
\gvac{2}\got{1}{V}\gnl
\gdb\gmp{\eta^{-2}_V}\gnl
\gcl{1}\gibr\gnl
\gcl{1}\gbr\gnl
\gcl{1}\gev\gnl
\gob{1}{V}
\gend
}}
=
{{\footnotesize
\gbeg{3}{7}
\gvac{2}\got{1}{V}\gnl
\gdb\gcl{1}\gnl
\gcl{1}\gibr\gnl
\gcl{1}\gmp{\eta^{-2}_V}\gcl{1}\gnl
\gcl{1}\gbr\gnl
\gcl{1}\gev\gnl
\gob{1}{V}
\gend
}}
\equal{\equref{etamorgen}}
{{\footnotesize
\gbeg{5}{8}
\gvac{4}\got{1}{V}\gnl
\gwdb{4}\gcl{1}\gnl
\gcl{1}\gdb\gibr\gnl
\gcl{1}\gbr\gcl{1}\gcl{2}\gnl
\gcl{1}\gcl{1}\gibr\gnl
\gcl{1}\gev\gbr\gnl
\gcl{1}\gvac{2}\gev\gnl
\gob{1}{V}
\gend
}}
=
{{\footnotesize
\gbeg{5}{8}
\gvac{4}\got{1}{V}\gnl
\gwdb{4}\gcl{1}\gnl
\gcl{1}\gdb\gcl{2}\gcl{4}\gnl
\gcl{1}\gbr\gnl
\gcl{1}\gcl{1}\gbr\gnl
\gcl{1}\gcl{1}\gev\gnl
\gcl{1}\gwev{4}\gnl
\gob{1}{V}
\gend
}},\eqlabel{impformribbonduality}
\end{eqnarray}
where in the last equality of the second computation we applied the naturality of $c^{-1}_{V,-}$ 
to the morphism ${\rm ev}_Vc_{V, V^*}: V\ot V^*\ra \un{1}$. 
These formulas allow to prove that 
\begin{eqnarray*}
{{\footnotesize
\gbeg{7}{8}
\gvac{6}\got{1}{V^*}\gnl
\gdb\gnot{\hspace*{-1.2cm}V^*}\gwdb{4}\gnot{\hspace*{-2.2cm}V}\gcl{1}\gnl
\gcl{1}\gibr\gdb\gnot{\hspace*{-1.2cm}}\gcl{1}\gcl{4}\gnl
\gev\gnot{\hspace*{-1.2cm}V}\gcl{1}\gbr\gcl{1}\gnl
\gvac{2}\gcl{1}\gcl{1}\gmp{\eta^2_V}\gcl{1}\gnl
\gvac{2}\gev\gnot{\hspace*{-1.2cm}V^*}\gbr\gnl
\gvac{4}\gcl{1}\gev\gnot{\hspace*{-1.2cm}V}\gnl
\gvac{4}\gob{1}{V^*}
\gend}}
&=&
{{\footnotesize
\gbeg{7}{8}
\gvac{6}\got{1}{V^*}\gnl
\gdb\gnot{\hspace*{-1.2cm}V^*}\gwdb{4}\gnot{\hspace*{-2.2cm}V}\gcl{1}\gnl
\gcl{1}\gcl{1}\gcl{1}\gdb\gnot{\hspace*{-1.2cm}V}\gcl{1}\gcl{1}\gnl
\gcl{1}\gcl{1}\gcl{1}\gbr\gcl{2}\gcl{3}\gnl
\gcl{1}\gcl{1}\gbr\gmp{\eta^2_V}\gnl
\gcl{1}\gev\gnot{\hspace*{-1.2cm}V^*}\gcl{1}\gbr\gnl
\gwev{4}\gnot{\hspace*{-2.2cm}V}\gcl{1}\gev\gnot{\hspace*{-1.2cm}V}\gnl
\gvac{4}\gob{1}{V^*}
\gend
}}
=
{{\footnotesize
\gbeg{5}{9}
\gvac{4}\got{1}{V^*}\gnl
\gwdb{4}\gnot{\hspace*{-2.2cm}V}\gcl{1}\gnl
\gcl{1}\gdb\gnot{\hspace*{-1.2cm}V}\gcl{3}\gcl{4}\gnl
\gcl{1}\gbr\gnl
\gcl{1}\gcl{1}\gmp{\eta^2_V}\gnl
\gbr\gbr\gnl
\gev\gnot{\hspace*{-1.2cm}V}\gcl{2}\gbr\gnl
\gvac{3}\gev\gnot{\hspace*{-1.2cm}V}\gnl
\gvac{2}\gob{1}{V^*}
\gend
}}\\
&=&
{{\footnotesize
\gbeg{3}{6}
\gvac{2}\got{1}{V^*}\gnl
\gdb\gnot{\hspace*{-1.2cm}V}\gcl{2}\gnl
\gbr\gnl
\gcl{1}\gbr\gnl
\gcl{1}\gev\gnot{\hspace*{-1.2cm}V}\gnl
\gob{1}{V^*}
\gend
}}
=
{{\footnotesize
\gbeg{3}{5}
\got{1}{V^*}\gnl
\gcl{1}\gdb\gnl
\gcl{1}\gmp{\eta^{-2}_V}\gcl{2}\gnl
\gev\gnl
\gvac{2}\gob{1}{V^*}
\gend
}}=(\eta^{-2}_V)^*=(\eta_{V^*})^{-2}.
\end{eqnarray*}
We used in the first equality an equivalent form of the naturality of $c^{-1}_{V,-}$ applied to ${\rm ev}_{V^*}$, 
the relation \equref{rrigid} in the second equality, \equref{rightdualribbon} in the third equality and in the fourth equality 
\equref{rrigid} and \equref{impformribbonduality}, respectively. In other words, the relation in \equref{etamorgen} 
is satisfied by $\eta_{V^*}$, and thus $(V^*, \eta_{V^*})$ is an object of ${\cal R}(\Cc, c)$. 
That it is a left dual of $(V, \eta_V)$ in ${\cal R}(\Cc, c)$ reduces to the fact that ${\rm ev}_V$ and ${\rm coev}_V$ are morphisms 
in ${\cal R}(\Cc, c)$. Towards this end, we need the equivalence  
\begin{equation}\eqlabel{rdualbasesribbon}
{{\footnotesize
\gbeg{4}{9}
\gvac{2}\got{1}{W^*}\got{1}{V}\gnl
\gdb\gcl{3}\gcl{5}\gnl
\gbr\gnl
\gcl{1}\gmp{\eta^2_V}\gnl
\gcl{1}\gbr\gnl
\gcl{1}\gcl{1}\gmp{f}\gcl{1}\gnl
\gcl{1}\gev\gcl{1}\gnl
\gwev{4}\gnl
\gvac{1}\gob{2}{\un{1}}
\gend}}
=
{{\footnotesize
\gbeg{4}{10}
\gvac{2}\got{1}{W^*}\got{1}{V}\gnl
\gdb\gcl{2}\gcl{2}\gnl
\gbr\gnl
\gcl{1}\gcl{1}\gibr\gnl
\gcl{1}\gibr\gcl{1}\gnl
\gev\gmp{\eta^2_V}\gcl{1}\gnl
\gvac{2}\gbr\gnl
\gvac{2}\gcl{1}\gmp{f}\gnl
\gvac{2}\gev\gnl
\gvac{1}\gob{2}{\un{1}}\gnl
\gend
}}
\equal{\equref{etamorgen}}
{{\footnotesize
\gbeg{2}{4}
\got{1}{W^*}\got{1}{V}\gnl
\gcl{1}\gmp{f}\gnl
\gev\gnl
\gob{2}{\un{1}}
\gend
}}~~\Longleftrightarrow~~
{{\footnotesize
\gbeg{3}{8}
\gvac{2}\got{1}{W^*}\gnl
\gdb\gcl{3}\gnl
\gbr\gnl
\gcl{1}\gmp{\eta^2_V}\gnl
\gcl{1}\gbr\gnl
\gcl{1}\gcl{1}\gmp{f}\gnl
\gcl{1}\gev\gnl
\gob{1}{V^*}
\gend}}=f^*,
\end{equation}   
true for any morphism $f:V\ra W$ in $\Cc$ with $(V, \eta_V)\in {\cal R}(\Cc, c)$, which follows from the naturality 
of $c^{-1}_{V,-}$ applied to ${\rm ev}_V(\Id_{V^*}\ot f\eta^2_V)c_{V, W^*}: V\ot W^*\ra \un{1}$ and the 
definition of $f^*$. 

Now, that ${\rm ev}_V$ is a morphism in ${\cal R}(\Cc, c)$ is a consequence of the computation 
\[
{{\footnotesize
\gbeg{2}{5}
\got{1}{V^*}\got{1}{V}\gnl
\gcl{1}\gcl{1}\gnl
\gsbox{2}\gnot{\hspace*{5mm}\eta_{V^*\ot V}}\gnl
\gev\gnl
\gob{2}{\un{1}}
\gend}}
=
{{\footnotesize
\gbeg{2}{6}
\got{1}{V^*}\got{1}{V}\gnl
\gbr\gnl
\gbr\gnl
\gmp{\eta_{V^*}}\gmp{\eta_V}\gnl
\gev\gnl
\gob{2}{\un{1}}
\gend
}}
=
{{\footnotesize
\gbeg{4}{7}
\got{1}{V^*}\gvac{2}\got{1}{V}\gnl
\gcl{1}\gdb\gcl{1}\gnl
\gbr\gev\gnl
\gbr\gnl
\gcl{1}\gmp{\eta^2_V}\gnl
\gev\gnl
\gob{2}{\un{1}}
\gend
}}
=
{{\footnotesize
\gbeg{4}{7}
\gvac{2}\got{1}{V^*}\got{1}{V}\gnl
\gdb\gcl{1}\gcl{1}\gnl
\gcl{1}\gibr\gcl{1}\gnl
\gbr\gev\gnl
\gcl{1}\gmp{\eta^2_V}\gnl
\gev\gnl
\gob{2}{\un{1}}
\gend}}
=
{{\footnotesize
\gbeg{4}{8}
\gvac{2}\got{1}{V^*}\got{1}{V}\gnl
\gdb\gcl{3}\gcl{5}\gnl
\gbr\gnl
\gcl{1}\gmp{\eta^2_V}\gnl
\gcl{1}\gbr\gnl
\gcl{1}\gev\gcl{1}\gnl
\gwev{4}\gnl
\gvac{1}\gob{2}{\un{1}}
\gend
}}
\equal{\equref{rdualbasesribbon}}
{\rm ev}_V.
\] 
Likewise, we compute that 
\[
{{\footnotesize
\gbeg{2}{5}
\got{2}{\un{1}}\gnl
\gdb\gnl
\gsbox{2}\gnot{\hspace*{5mm}\eta_{V\ot V^*}}\gnl
\gcl{1}\gcl{1}\gnl
\gob{1}{V}\gob{1}{V^*}
\gend
}}
=
{{\footnotesize
\gbeg{2}{6}
\got{2}{\un{1}}\gnl
\gdb\gnl
\gbr\gnl
\gbr\gnl
\gmp{\eta_V}\gmp{\eta_{V^*}}\gnl
\gob{1}{V}\gob{1}{V^*}
\gend
}}
=
{{\footnotesize
\gbeg{4}{6}
\gvac{1}\got{2}{\un{1}}\gnl
\gdb\gnl
\gbr\gdb\gnl
\gbr\gmp{\eta_V}\gcl{1}\gnl
\gmp{\eta_V}\gev\gcl{1}\gnl
\gob{1}{V}\gvac{2}\gob{1}{V^*}
\gend
}}
\equal{\equref{etamorgen}}
{\rm coev}_V,
\]
and so ${\rm coev}_V$ is a morphism in ${\cal R}(\Cc, c)$, as stated. 

Finally, from the left rigid monoidal structure of ${\cal R}(\Cc, c)$ we get that 
$\eta:=(\eta_V)_V$ is a twist on ${\cal R}(\Cc, c)$, and so the proof is finished.  
\end{proof}

\thref{ribbonfrombraided} allows to construct ribbon categories from left or right rigid monoidal 
categories.  

\begin{proposition}\prlabel{prdefrlC}
Let $\Cc$ be a left rigid (strict) monoidal category. Then ${\mf R}_l(\Cc)$ is a category whose  

$\bullet$ objects are triples $(V, c_{V, -}, \eta_V)$ consisting of an object $V$ of $\Cc$, a natural 
isomorphism $c_{V, -}=\left(c_{V, X}: V\ot X\ra X\ot V\right)_{X\in {\rm Ob}(\Cc)}$ and an automorphism $\eta_V$ of $V$ in $\Cc$ 
such that 
\begin{equation}\eqlabel{center1} 
c_{V, X\ot Y}=(\Id_X\ot c_{V, Y})(c_{V, X}\ot \Id_Y),~\forall~X,~Y\in \Cc,
\end{equation}
\equref{msquaretwist} holds, and 
\begin{equation}
(\Id_X\ot \eta_V)c_{V, X}=c_{V, X}(\eta_V\ot \Id_X)~,~\forall~X\in {\rm Ob}(\Cc);\eqlabel{ribbon3}
\end{equation}    

$\bullet$ morphisms $f: (V, c_{V,-}, \eta_V)\ra (V', c_{V',-}, \eta_{V'})$ are morphisms 
$f:V\ra V'$ in $\Cc$ for which $(\Id_X\ot f)c_{V, X}=c_{V', X}(f\ot \Id_X)$, for any object $X$ of $\Cc$, 
and $f\eta_V=\eta_{V'}f$. 
\end{proposition}
\begin{proof}
One can easily check that ${\mf R}_l(\Cc)={\cal R}({\cal Z}_l(\Cc), c)$, where ${\cal Z}_l(\Cc)$ is the left center 
of $\Cc$ as in \cite{bcp}, a braided category. We only notice that \equref{msquaretwist} 
is nothing but the second equality in \equref{etamorgen}.  
\end{proof}

We now uncover the ribbon structure of ${\cal R}_l(\Cc)$, the category denoted by $\Dc(\Cc)$ in \cite{kasturaev}. 
For the choice of the natural isomorphism $c_{V^*,-}$ below see \cite[Proposition 7.4]{jsp} or \cite{yetter}. 

\begin{theorem}\thlabel{ribbonconstr}
Let $\Cc$ be a left rigid monoidal category. Then ${\cal R}_l(\Cc)$ is a ribbon category with the following structure: 

$\bullet$ If $(V, c_{V,-}, \eta_V), (W, c_{W, -}, \eta_W)\in {\cal R}_l(\Cc)$ then 
\begin{equation}\eqlabel{monstrribbconstr}
(V, c_{V,-}, \eta_V)\ot (W, c_{W, -}, \eta_W)=(V\ot W, c_{V\ot W, -}, \eta_{V\ot W}),
\end{equation}
where $c_{V\ot W, -}$ and $\eta_{V\ot W}$ are respectively given by 
\begin{eqnarray}
&&c_{V\ot W, X}=(c_{V, X}\ot \Id_W)(\Id_V\ot c_{W, X}),~\forall~X\in \Cc,\eqlabel{center2}\\ 
&&\eta_{V\ot W}=(\eta_V\ot\eta_W)c_{W,V}c_{V,W};\eqlabel{etamonoidal}
\end{eqnarray} 

$\bullet$ The unit object is $(\un{1}, c_{\un{1}, -}=(r^{-1}_Xl_X)_{X\in \Cc}\equiv \Id, \eta_{\un{1}}=\Id_{\un{1}})$, and the 
associativity and the left and right unit constraints are the same as those of $\Cc$;

$\bullet$ The braiding $\mf{c}$ is determined by 
\begin{equation}
c_{V,W}: (V, c_{V, -}, \eta_V)\ot (W, c_{W, -}, \eta_W)\ra (W, c_{W, -}, \eta_W)\ot (V, c_{V, -}, \eta_V),\eqlabel{braidingribbconstr} 
\end{equation}
an isomorphism in ${\cal R}_l(\Cc)$; 

$\bullet$ For $(V, c_{V, -}, \eta_V)\in {\cal R}_l(\Cc)$, a left dual object for it is $(V^*, c_{V^*,-}, \eta_{V^*})$, where
\begin{equation}
c_{V^*, X}=({\rm ev}_V\ot \Id_{X\ot V^*})(\Id_{V^*}\ot c^{-1}_{V, X}\ot \Id_{V^*})(\Id_{V^*\ot X}\ot {\rm coev}_V),\eqlabel{braiddualobjribb}
\end{equation}
for all $X\in {\rm Ob}(\Cc)$, and 
\begin{equation}
\eta_{V^*}=(\eta_V)^*.\eqlabel{twistdualobjribb}
\end{equation}
The evaluation and coevaluation morphisms are ${\rm ev}_V$ and ${\rm coev}_V$, viewed now as morphisms in ${\cal R}_l(\Cc)$; 

$\bullet$ The twist is given by 
\begin{equation}
\eta_V: (V, c_{V, -}, \eta_V)\ra (V, c_{V, -}, \eta_V),\eqlabel{twistribbonconstr}
\end{equation} 
an automorphism in ${\cal R}_l(\Cc)$. 
\end{theorem}
\begin{proof}
Since ${\mf R}_l(\Cc)={\cal R}({\cal Z}_l(\Cc), c)$, the only thing we must check is the fact 
${\cal Z}_l(\Cc)$ is left rigid, provided that $\Cc$ is so. To this end, according to \cite[Proposition 3]{streeqd} it suffices to show 
that $c_{V^*, X}$ in \equref{braiddualobjribb} is an isomorphism in $\Cc$, 
for all $X\in {\rm Ob}(\Cc)$, since this will imply that $(V^*, c_{V^*, -})$ is a left dual of $(V, c_{V,-})$ in ${\cal Z}_l(\Cc)$. 
But this fact was proved in \cite[Lemma 4.1]{kasturaev}, by showing that     
\[
c_{V^*, X}^{-1}=
{{\footnotesize
\gbeg{4}{5}
\gvac{2}\got{1}{X}\got{1}{V^*}\gnl
\gdbd\gvac{2}\gcl{1}\gcl{1}\gnl
\gcl{1}\gbrb\gvac{2}\gcl{1}\gnl
\gcl{1}\gcl{1}\gevd\gnl
\gob{1}{V^*}\gob{1}{X}
\gend
}},
~~\mbox{with}~~
c_{V, X}:=
{{\footnotesize
\gbeg{2}{3}
\got{1}{V}\got{1}{X}\gnl
\gbrb\gnl
\gob{1}{X}\gob{1}{V}
\gend}},
\] 
is the inverse of $c_{V^*, X}$ in $\Cc$, for all $X\in \Cc$. Here, as the graphical notation suggests, the evaluation and coevaluation morphisms 
with a black dot are ${\rm ev}'_V$ and ${\rm coev}'_V$ defined as in \equref{rightdualitysovereign} and 
\equref{rightdualitysovereign2coev}, respectively, of course with $c_{V, V^*}$ replaced by the component $V^*$ of our $c_{V, -}$.  
\end{proof}

\begin{definition}
If $(\Cc, c, \eta)$ and $({\cal D}, d, \theta)$ are ribbon categories then a functor $F: \Cc\ra {\cal D}$ is called a ribbon 
functor if it is a braided monoidal functor such that $F(\eta_V)=\theta_{F(V)}$, for all 
$V\in {\rm Ob}(\Cc)$, and is compatible with the duality.  
\end{definition}

The left handed version of the universal property of the centre (see \cite[Proposition XIII.4.3]{kas}) leads to a similar property 
for ${\cal R}_l(\Cc)$; see \cite[Theorem 2.5]{kasturaev}. It can be derived also from the universal property of 
our ${\cal R}(\Cc, c)$ that we prove below.  Indeed, the universal property of the (left) centre assigns (under some conditions) 
to a monoidal functor $F$ from a braided category $\Dc$ to a monoidal category $\Cc$ a braided monoidal functor ${\cal Z}_l(F)$  
from $\Dc$ to ${\cal Z}_l(\Cc)$, the (left) centre of $\Cc$. Then \cite[Proposition XIII.4.3]{kas} follows by applying our 
universal property to the braided functor ${\cal Z}_l(F)$.  

\begin{proposition}\prlabel{UnivRibbonCateg}
Let $({\cal D}, d, \theta)$ be a ribbon category, $(\Cc, c)$ a left rigid braided category and $F: ({\cal D}, d)\ra (\Cc, c)$ 
a braided monoidal functor. Then there exists a unique 
ribbon functor ${\cal R}(F): {\cal D}\ra {\cal R}(\Cc)$ such that $\Pi\circ {\cal R}(F)=F$, where 
$\Pi: {\cal R}(\Cc)\ra \Cc$ is the functor that forgets the ribbon twist.  
\end{proposition}
\begin{proof}
If there exists a ribbon functor ${\cal R}(F): {\cal D}\ra {\cal R}(\Cc)$ such that $\Pi\circ {\cal R}(F)=F$ it 
follows that ${\cal R}(F)(X)=(F(X), F(\theta_X))$, for any object $X$ of $\Dc$, and that 
${\cal R}(F)(f)=F(f)$, for any morphism $f$ in $\Dc$. This proves the uniqueness of ${\cal R}(F)$.  

Conversely, define ${\cal R}(F)(X)=(F(X), F(\theta_X))$, for any object $X$ of $\Dc$, and for a morphism $f$ in 
$\Dc$ set ${\cal R}(F)(f)=F(f)$. Since $F$ is strong monoidal it respects the left dualities on $\Dc$ and $\Cc$. By using the 
uniqueness (up to isomorphism) of a left dual object we can assume without loss of generality that $F(X^\sharp)=F(X)^*$, 
where $X^\sharp$ is the left dual object of $X$ in $\Dc$ and $F(X)^*$ is the left dual of $F(X)$ in $\Cc$. 
Then the evaluation and coevaluation 
morphisms for the adjunction $F(X)^*\dashv F(X)$ in $\Cc$ are obtained from those of $X^\sharp\dashv X$ in $\Dc$ and the 
strong monoidal structure of the functor $F$. But $F$ is, moreover, braided monoidal and together with the above 
arguments and the fact that $\theta$ verifies \equref{msquaretwist} this implies that $\eta_{F(X)}:=F(\theta_X)$ 
verifies \equref{etamorgen}. In other words ${\cal R}(F)$ is well defined on objects. It is well defined on morphisms, too, 
since $\theta$ is a natural isomorphism.     
\end{proof} 

\begin{corollary}\colabel{ribboncaseforcentrefunctor}
If $(\Cc, c, \eta)$ is a ribbon category then there exists a unique 
ribbon functor ${\cal R}: \Cc\ra {\cal R}(\Cc, c)$ such that 
$\Pi\circ {\cal R}=\Id_\Cc$. 
\end{corollary}
\begin{proof}
Take ${\cal D}=\Cc$ and $F=\Id_\Cc$ in \prref{UnivRibbonCateg}. Then 
${\cal R}={\cal R}(\Id_\Cc)$. 
\end{proof}

In what follows we also need the right handed version of ${\cal R}_l(\Cc)$. In fact, if we start with a right rigid monoidal 
category then $\ov{\Cc}$ is a left rigid monoidal category, and so we can consider ${\cal R}_l(\ov{\Cc})$. 
Thus ${\cal R}_r(\Cc):=\ov{{\cal R}_l(\ov{\Cc})}$ is a ribbon category, too. More precisely, we have the following.

\begin{proposition}
Let $\Cc$ be a right rigid monoidal category. Then the objects of ${\cal R}_r(\Cc)$ are triples 
$(V, c_{-,V}, \theta_V)$ consisting of an object $V$ of $\Cc$, a natural isomorphism 
$c_{-,V}=(c_{X, V}: X\ot V\ra V\ot X)_{X\in {\rm Ob}(\Cc)}$ and a morphism $\theta_V: V\ra V$ in $\Cc$, subject  
to the following conditions:

$\bullet$ $c_{\un{1}, V}\equiv \Id_V$ and $c_{X\ot Y, V}=(c_{X, V}\ot \Id_Y)(\Id_X\ot c_{Y, V})$, for all $X, Y\in \Cc$;

$\bullet$ $\theta_V$ is an automorphism of $V$ in $\Cc$ obeying $c_{X, V}(\Id_X\ot \theta_V)=(\theta_V\ot \Id_X)c_{X, V}$, 
for all $X\in \Cc$, and 
\begin{equation}\eqlabel{condetaforRrightC} 
\theta^{-2}_V:=(\Id_V\ot {\rm ev}'_V)(c^{-1}_{V, V}\ot\Id_{{}^*V})(\Id_V\ot c_{{}^*V,V})(\Id_V\ot {\rm coev}'_V). 
\end{equation}

A morphism $f: (V, c_{-,V}, \theta_V)\ra (W, c_{-,W}, \theta_W)$ in ${\cal Z}_r(\Cc)$ is a morphism 
$f:V\ra W$ in $\Cc$ such that $c_{X,W}(f\ot \Id_X)=(\Id_X\ot f)c_{X, V}$, for all $X\in \Cc$. 

The category ${\cal R}_r(\Cc)$ is ribbon via the following structure:

$\bullet$ the tensor product of $(V, c_{-,V}, \theta_V)$ and $(W, c_{-,W}, \theta_W)$ in ${\cal R}_r(\Cc)$ 
is $(V\ot W, c_{-, V\ot W}, \theta_{V\ot W})$, where 
$
c_{X, V\ot W}=(\Id_V\ot c_{X, W})(c_{X, V}\ot \Id_W)
$, for all $X\in \Cc$, 
and $\theta_{V\ot W}=(\theta_V\ot \theta_W)c_{W,V}c_{V,W}$, and the tensor product of two morphisms in ${\cal R}_r(\Cc)$ 
is their tensor product in $\Cc$, while the associativity and the left and right unit 
constraints are the same as those of $\Cc$;

$\bullet$ the braiding between two objects $(V, c_{-,V}, \theta_V)$ and $(W, c_{-,W}, \theta_W)$ in ${\cal R}_r(\Cc)$ 
is given by $c_{V, W}$;

$\bullet$ the right dual object of $(V, c_{-, V}, \theta_V)$ in ${\cal R}_r(\Cc)$ is $({}^*V, c_{-, {}^*V}, \theta_{{}^*V})$ 
determined by 
\[
c_{X, {}^*V}=(\Id_{{}^*V\ot X}\ot {\rm ev}'_V)(\Id_{V^*}\ot c^{-1}_{V, X}\ot \Id_{{}^*V})({\rm coev}'_V\ot \Id_{X\ot {}^*V}), 
\]
for all $X\in \Cc$, $\theta_{{}^*V}={}^*(\theta_V)$, and the evaluation and coevaluation morphisms equal 
${\rm ev}'_V$ and ${\rm coev}'_V$, respectively;

$\bullet$ the twist is given by $\theta_V: (V, c_{-, V}, \theta_V)\ra (V, c_{-, V}, \theta_V)$, for all $V\in \Cc$. 
\end{proposition}  
\begin{proof}
As ${\cal R}_r(\Cc):=\ov{{\cal R}_l(\ov{\Cc})}$, everything follows from the above comments and results.
\end{proof}
\subsection{A class of ribbon quasi-Hopf algebras}\selabel{specexpribbqHA}
In general, $H[\theta, \theta^{-1}]$ considered in \sseref{balancedquasiHopf} is not a quasi-Hopf algebra, and so neither a ribbon 
quasi-Hopf algebra. To "make" it ribbon, we have to consider a quotient of it. Actually,  
we have to consider the $k$-algebra $H(\theta):=\frac{H[\theta]}{\le \theta^2 - uS(u)\ri}$ instead 
of $H[\theta, \theta^{-1}]$, where $H[\theta]$ is 
the free $k$-algebra generated by $H$ and $\theta$ with relations $h\theta=\theta h$, for all $h\in H$,  
and $u$ is as in \equref{elmu}. 

In what follows, we still denote by $\theta$ the class in $H(\theta)$ of $\theta$. 

\begin{proposition}\prlabel{ribbonquasiHopffromQT} 
If $(H, R)$ is a QT quasi-Hopf algebra then $H(\theta)$ is a quasi-Hopf algebra with structure determined by 
$\Delta\mid_H=\Delta_H$, $\va\mid_H=\va_H$, $S\mid_H=S$, 
\[
\Delta(\theta)=(\theta\ot \theta)(R_{21}R)^{-1}~,~\va(\theta)=1~,~S(\theta)=\theta,
\] 
and the reassociator and distinguished elements that define the antipode equal to those of $H$. 
Furthermore, $(H(\theta), R)$ is QT and $\theta^{-1}$ defines a ribbon twist on ${}_{H(\theta)}{\cal M}^{\rm fd}$ 
as in \equref{etaforribbonqHa}. 
\end{proposition}
\begin{proof}
One can see easily that $\Delta(h)R_{21}R=R_{21}R\Delta(h)$, for all $h\in H$. So by \equref{qrib7} and 
\equref{qrib6} one compute that 
\[
\Delta(uS(u))=f^{-1}(S\ot S)(f_{21})(uS(u)\ot uS(u))(S\ot S)(f_{21}^{-1})f(R_{21}R)^{-2},
\] 
and because $uS(u)$ is central in $H$ we get $\Delta(uS(u))=(uS(u)\ot uS(u))(R_{21}R)^{-2}$. 
This shows that $\Delta$ is well defined on $\theta$. Also, it follows from 
\prref{ribbquasibialgfromQTones} that $H(\theta)$ is a quasi-bialgebra. 

Furthermore, since $S(\theta^2-uS(u))=S(\theta)^2-S^2(u)S(u)=\theta^2-uS(u)$ we deduce that 
$S$ is well defined on $\theta$, too. So it remains to show the equalities 
\[
S(\theta_1)\a\theta_2=\alpha~~\mbox{and}~~\theta_1\b S(\theta_2)=\b~.
\] 

The equality in \equref{extr} can be rewritten as $S(\a\ov{R}^2)u\ov{R}^1=\a$ or, equivalently, 
as $S(\ov{R}^1)\a \ov{R}^2=S^{-1}(\a u^{-1})=S(u^{-1}\a )$, cf. \equref{ssinau}. Hence 
\begin{eqnarray*}
S(\theta_1)\a \theta_2&=&
S(\ov{R}^1\ov{r}^2)\a \ov{R}^2\ov{r}^1\theta^2\\
&=&S(u^{-1}\a \ov{r}^2)\ov{r}^1\theta^2\\
&\equal{\equref{ssinau}}&S(u^{-1}S(\ov{r}^1)\a \ov{r}^2)\theta^2\\
&=&S^2(u^{-1}\a)S(u^{-1})\theta^2\\
&\equal{\equref{ssinau}}&\a u^{-1}S(u^{-1})\theta^2=\a, 
\end{eqnarray*}
as required; in the last equality we used that $uS(u)=S(u)u$.  

Notice that the formula in (\ref{fu18}) is equivalent to $\ov{R}^1S(\ov{R}^2\b u)=\b$, 
and therefore with $\ov{R}^2\b S(\ov{R}^1)=u^{-1}S(\b)$, because of \equref{ssinau}. The latter gives us 
\begin{eqnarray*}
\theta_1\b S(\theta_2)&=&\ov{R}^1\ov{r}^2\b S(\ov{R}^2\ov{r}^1)\theta^2\\
&=&\ov{R}^1u^{-1}S(\ov{R}^2\b)\theta^2\\
&\equal{\equref{ssinau}}&
u^{-1}S(\ov{R}^2\b S(\ov{R}^1))\theta^2\\
&=&u^{-1}S(u^{-1}S(\b))\theta^2\\
&\equal{\equref{ssinau}}&
u^{-1}S(u^{-1})\theta^2\b=\b,
\end{eqnarray*} 
since $uS(u)=S(u)u$. Finally, $R$ is an $R$-matrix for $H(\theta)$ because so is for $H$ and 
\begin{eqnarray*}
R\Delta(\theta)&=&R(R_{21}R)^{-1}(\theta\ot \theta)\\
&=&R_{21}^{-1}(\theta\ot \theta)\\
&=&(RR_{21})^{-1}R(\theta\ot \theta)
=\Delta^{\rm cop}(\theta)R.
\end{eqnarray*}
Clearly $\theta^{-1}$ defines a ribbon structure for $(H(\theta), R)$, and this completes the proof.  
\end{proof}

Our next goal is to show that the category of finite dimensional left modules over $H(\theta)$ 
can be identified with the category ${\cal R}({}_H{\cal M}^{\rm fd}, c)$, where $c$ is defined by \equref{braiQT}. 

\begin{theorem}\thlabel{qtreprasrepresoverribbon}
Let $(H, R)$ be a QT quasi-Hopf algebra, so ${}_H{\cal M}^{\rm fd}$ is a rigid braided category. 
Then ${\cal R}({}_H{\cal M}^{\rm fd}, c)$ identifies as a ribbon category with ${}_{H(\theta)}{\cal M}^{\rm fd}$.  
\end{theorem}
\begin{proof}
By taking $\Cc={}_H{\cal M}^{\rm fd}$ in \deref{ribbfrombraidwithld}, we deduce that an object of 
the category ${\cal R}({}_H{\cal M}^{\rm fd}, c)$ is a pair $(V, \eta_V)$ consisting of a finite dimensional 
left $H$-module $V$ and an $H$-automorphism $\eta_V$ of $V$ such that 
\begin{eqnarray*}
\eta^{-2}(v)&=&X^1r^2R^1\b S(X^2r^1R^2)\a X^3\cdot v\\
&=&q^1r^2R^1\b S(q^2r^1R^2)\cdot v\\
&\equal{(\ref{fu18})}&S(q^2r^1\b uS^{-1}(q^1r^2))\cdot v\\
&\equal{\equref{ssinau}}&S(q^2r^1\b S(q^1r^2)u)\cdot v\\
&\equal{(\ref{fu18})}&S(q^2S(q^1\b u)u)\cdot v\\
&\equal{\equref{ssinau}}&S(u)S^2(q^1\b S(q^2)u)\cdot v\\
&\equal{\equref{q6}}&S(u)u\cdot v=uS(u)\cdot v,  
\end{eqnarray*} 
for all $v\in V$. Thus objects of ${\cal R}({}_H{\cal M}^{\rm fd}, c)$ are pairs $(V, \eta_V)$ 
consisting of a finite dimensional left $H$-module and an $H$-automorphism $\eta_V$ of $V$ 
satisfying $\eta^2(v)=(uS(u))^{-1}\cdot v$, for all $v\in V$. 

Clearly, a morphism $f: (V, \eta_V)\ra (W, \eta_W)$ is a left $H$-linear morphism $f: V\ra W$ 
such that $\eta_Wf=f\eta_V$. 

Let $F: {\cal R}({}_H{\cal M}^{\rm fd}, c)\ra {}_{H(\theta)}{\cal M}^{\rm fd}$ be the functor defined 
as follows: $F(V, \eta_V)=V$ regarded as $H(\theta)$-module via the $H$-action on $V$ and 
$\theta\cdot v=\eta^{-1}_V(v)$; $F$ acts as identity on morphisms. 

We can easily see that $\theta^2=uS(u)$ together with  $\eta^{-2}(v)=(uS(u))^{-1}\cdot v$, for all $v\in V$, 
implies that $F$ is a well defined functor. It provides an isomorphism of categories, its inverse being the  
functor $G: {}_{H(\theta)}{\cal M}^{\rm fd}\ra {\cal R}({}_H{\cal M}^{\rm fd}, c)$ given 
by $G(V)=(V, \eta_V: V\ni v\mapsto \theta^{-1}\cdot v\in V)$, for all $V\in {}_{H(\theta)}{\cal M}^{\rm fd}$.   

The functor $F$ is monoidal since 
\begin{eqnarray*}
\theta\cdot (v\ot w)&=&\eta^{-1}_{V\ot W}(v\ot w)\\
&=&(R_{21}R)^{-1}(\eta^{-1}_V(v)\ot \eta^{-1}_W(w))\\
&=&(R_{21}R)^{-1}(\theta\cdot v\ot \theta\cdot w)
=\theta_1\cdot v\ot \theta_2\cdot w,
\end{eqnarray*}  
for all $V, W\in {\cal R}({}_H{\cal M}^{\rm fd}, c)$, $v\in V$ and $w\in W$. Furthermore, $F$ 
is braided because for both categories the braiding is defined by $c$. 

The ribbon structure $\widetilde{\eta}$ of ${}_{H(\theta)}{\cal M}^{\rm fd}$ is induced by the element 
$\theta^{-1}$, in the sense that $\widetilde{\eta}_V(v)=\theta^{-1}\cdot v=\eta_V(v)$, for all 
$v\in V\in {}_{H(\theta)}{\cal M}^{\rm fd}$. Otherwise stated, the two categories have the same 
ribbon structure, and therefore $F$ is a ribbon isomorphism functor. 
Observe also that $F$ is compatible with the left rigid monoidal structures on  
${\cal R}({}_H{\cal M}^{\rm fd}, c)$ and ${}_{H(\theta)}{\cal M}^{\rm fd}$. More precisely, we have 
$F(V^*, \eta_{V^*})=V^*$ with left $H(\theta)$-module structure induced by that of the left $H$-module 
and $\theta\cdot v^*=\eta^{-1}_{V^*}(v^*)=(\eta^{-1}_V)^*(v^*)=v^*\circ \eta^{-1}_V=v^*(\theta\cdot)
=v^*(S(\theta)\cdot)$, for all $v^*\in V^*$, as required.  
\end{proof}

\begin{remark}
If $(H, R)$ is QT then
$\tilde{R}=R^{-1}_{21}=\ov{R}^2\ot \ov{R}^1$ is another
$R$-matrix for $H$. In addition, if $\tilde{u}$ is the element 
\equref{elmu} corresponding to $(H, \tilde{R})$ then by the relation (6.23) in \cite{btfact} we have that $\tilde{u}=S(u^{-1})$, 
so $\tilde{u}S(\tilde{u})=(uS(u))^{-1}$. Therefore, if $\tilde{c}$ is the braiding 
on ${}_H{\cal M}^{\rm fd}$ defined by $\tilde{R}$ then ${\cal R}({}_H{\cal M}^{\rm fd}, \tilde{c})$ 
is a ribbon category that is isomorphic to ${}_{H(\theta)}{\cal M}^{\rm fd}$, too. 
To see this, observe that an object of ${\cal R}({}_H{\cal M}^{\rm fd}, \tilde{c})$ 
is a pair $(V, \theta_V)$ consisting of $V\in {}_H{\cal M}^{\rm fd}$ and an automorphism 
$\theta_V$ of the left $H$-module $V$ such that $\theta^2_V(v)=uS(u)\cdot v$, for all $v\in V$. 
It is clear at this point that $(V, \eta_V)\mapsto (V, \theta_V:=\eta_V^{-1})$ 
defines the desired isomorphism of categories. 
\end{remark}

\begin{corollary}\colabel{ribbdrinffromadj}
Let $H$ be a finite dimensional quasi-Hopf algebra and $D(H)$ its quantum double. 
If $D(H, \theta):=D(H)(\theta)$ then 
\[
{\mf R}_r({}_H{\cal M}^{\rm fd})\simeq {\cal R}({}_H{\YD}^{H{\rm fd}}, {\mf c})\simeq 
{}_{D(H, \theta)}{\cal M}^{\rm fd} 
\]
as ribbon categories, where ${\mf c}$ is the braiding defined in \equref{lryd4}.  
\end{corollary}
\begin{proof}
The first isomorphism can be deduced from the definition 
${\mf R}_r({}_H{\cal M}^{\rm fd})={\cal R}({\cal Z}_r({}_H{\cal M}^{\rm fd}), c)$ and the 
braided isomorphism between ${\cal Z}_r({}_H{\cal M}^{\rm fd})$ and ${}_H{\YD}^{H{\rm fd}}$ 
established by \cite[Theorem 2.2]{bcp}. The second one follows from 
the braided isomorphism $({}_H{\YD}^{H{\rm fd}}, {\mf c})\cong ({}_{D(H)}{\cal M}^{\rm fd}, {\cal R}_D)$ 
established in \cite[Proposition 3.1]{btThRank}, and \thref{qtreprasrepresoverribbon}. 
\end{proof}



\begin{thebibliography}{99}
\bibitem{barret1}
J. W. Barret, B. W. Westbury, {\sl Spherical categories}, Adv. Math. {\bf 143} (1999), 357--375.

\bibitem{bc1}
D. Bulacu, S. Caenepeel, {\sl Integrals for (dual) quasi-Hopf algebras. Applications}, 
J. Algebra {\bf 266}(2) (2003), 552--583.

\bibitem{bcp}
D. Bulacu, S. Caenepeel, F. Panaite, {\sl Yetter-Drinfeld categories
for quasi-Hopf algebras}, Comm. Algebra {\bf 34} (2006), 1--35.

\bibitem{bctInv}
D. Bulacu, S. Caenepeel, B. Torrecillas, {\sl Involutory quasi-Hopf algebras}, 
Algebr. Represent. Theory {\bf 12} (2009), 257--285.

\bibitem{bctKlein}
D. Bulacu, S. Caenepeel, B. Torrecillas, {\sl The braided monoidal structures on the category of vector 
spaces graded by the Klein group}, Proc. Edinburgh Math. Soc. {\bf 54} (2011), 613--641. 

\bibitem{bn3}
D. Bulacu, E. Nauwelaerts, {\sl Quasitriangular and ribbon quasi-Hopf algebras}, 
Comm. Algebra {\bf 31}(2) (2003), 657--672.

\bibitem{bpv}
D. Bulacu, F. Panaite, F. Van Oystaeyen, {\sl Quantum traces and quantum dimensions
for quasi-Hopf algebras}, Comm. Algebra {\bf 27} (1999), 6103–-6122.

\bibitem{bpvo}
D, Bulacu, F. Panaite, F. Van Oystaeyen, {\sl Generalized diagonal crossed 
products and smash products for quasi-Hopf algebras. Applications}, Comm. Math. Phys. 
{\bf 266} (2006), 355--399.

\bibitem{btfact}
D. Bulacu, B. Torrecillas, {\sl Factorizable quasi-Hopf algebras. Applications}, 
J. Pure Appl. Algebra {\bf 194} (2004), 39--84.

\bibitem{btThRank}
D. Bulacu, B. Torrecillas, {\sl The representation-theoretic 
rank of the doubles of quasi-quantum groups}, J. Pure Appl. Algebra, {\bf 212} (2008), 919-–940.

\bibitem{btRecon}
D. Bulacu, B. Torrecillas, {\sl Quasi-quantum groups obtained from the Tannaka-Krein reconstruction theorem}, 
preprint 2018, submitted. 

\bibitem{deligne}
P. Deligne, {\sl Cat\'egories Tannakiennes}, The Grothendieck Festschrift, Vol. 2, 
Birkh\"auser 1990, p. 111--195.

\bibitem{drabant}
B. Drabant, {\sl Notes on balanced categories and Hopf algebras}, "Quantum Groups and Lie Theory" (Durham, 1999), 
Cambridge University Press, Cambridge 2001, 63--88.

\bibitem{dri}
V. G. Drinfeld, {\sl Quasi-Hopf algebras}, Leningrad Math. J. {\bf 1} (1990), 1419--1457.

\bibitem{eno}
P. Etingof, D. Nikshych, V. Ostrik, {\sl On fusion categories}, 
Ann. of Math. {\bf 162} (2) (2005), 581--642.

\bibitem{egnobook}
P. Etingof, S. Gelaki, D. Nikshych, V. Ostrik, {\sl Tensor categories},  
Mathematical Surveys and Monographs, {\bf 205}, Amer. Math. Soc., Providence, RI, 2015.  

\bibitem{feYett}
P. J. Freyd, D. N. Yetter, {\sl Coherence theorems via knot theory}, J. Pure Appl. Alg. {\bf 78} (1992), 49--76.

\bibitem{gel}
S. Gelaki, {\sl Basic quasi-Hopf algebras of dimension $n^3$}, J. Pure Appl. Algebra {\bf 198} (2005) , 165--174.

\bibitem{hnRMP}
F. Hausser and F. Nill, {\sl Diagonal crossed products by duals of quasi-quantum
groups}, Rev. Math. Phys. {\bf 11} (1999), 553--629.

\bibitem{henr}
A. Henriques, D. Penneys, J. Tener, {\sl Categorified trace for module tensor categories over braided tensor categories}, 
Documenta Mathematica {\bf 21} (2016), 1089–-1149.

\bibitem{jscentre}
A. Joyal, R. Street, {\sl Tortile Yang-Baxter operators in tensor categories}, J.
Pure Appl. Algebra {\bf 71} (1991), 43-–51. 

\bibitem{jsp} 
A. Joyal, R. Street, {\sl Braided tensor categories}, {\em Adv. Math.} {\bf 102} (1993), 20--78.

\bibitem{kasturaev}
C. Kassel, V. G. Turaev, {\sl Double construction for monoidal categories}, Acta Math. {\bf 175} (1995), 1--48.

\bibitem{kas}
C. Kassel, {\sl Quantum Groups}, Graduate Texts in Mathematics {\bf 155}, Springer Verlag, 
Berlin, 1995.

\bibitem{lr1}
R. G. Larson, D. E. Radford, {\sl Finite-dimensional cosemisimple Hopf algebras
in characteristic $0$ are semisimple}, J. Algebra {\bf 117} (1988), 267--289.

\bibitem{majcentre}
S. Majid, {\sl Representations, duals and quantum doubles of monoidal categories}, 
Rend. Circ. Math. Palermo (2) Suppl. {\bf 26} (1991), 197-–206.

\bibitem{maj}
S. Majid, {\sl Foundations of quantum group theory}, Cambridge University Press, 1995.

\bibitem{mn}
G. Mason, Siu-Hung Ng, {\sl Central invariants and Frobenius-Schur indicators for 
semisimple quasi-Hopf algebras}, Adv. Math. {\bf 190} (2005), no. 1, 161--195. 

\bibitem{pan}
F. Panaite, {\sl A Maschke-type theorem for quasi-Hopf algebras}, Rings, Hopf algebras and Brauer groups (Antwerp/Brussels, 1996), 
Lecture Notes in Pure Appl. Math., Marcel Dekker, New York, {\bf 197} (1998), 201--207. 

\bibitem{hp}
H. Pfeiffer, {\sl Tannaka-Krein reconstruction and a characterization of modular tensor categories}, J. Algebra {\bf 321} (2009), 
3714--3763. 

\bibitem{rt}
N. Y. Reshetikhin, V. G. Turaev, {\sl Modular categories and $3$-manifold invariants}, Internat. J. Modern Phys. B {\bf 6} (1992), 
1807--1824. 

\bibitem{streeqd}
R. Street, {\sl The quantum double and related constructions}, J. Pure Appl. Algebra {\bf 132} (1999), 195--206.

\bibitem{sch}
P. Schauenburg, {\sl On the Frobenius-Schur indicators for quasi-Hopf algebras}, 
{\sl J. Algebra} {\bf 282} (2004), no. 1, 129--139.

\bibitem{turaev}
V. G. Turaev, {\sl Modular categories and $3$-manifold invariants}, Internat. J. Modern Phys. B 
{\bf 6} (1992), 1807--1824.  

\bibitem{yetter}
D. N. Yetter, {\sl Quantum groups and representations of monoidal categories}, 
Math. Proc. Cambridge Philos. Soc. {\bf 108} (1990), 261--290.
\end{thebibliography}
\end{document}